\documentclass[12 pt, reqno]{amsart}
\usepackage{fullpage}
\usepackage{amsthm,amssymb,amsfonts,hyperref, mathrsfs}
\usepackage{url}
\hypersetup{
    colorlinks=true,       % false: boxed links; true: colored links
    linkcolor=blue,          % color of internal links (change box color with linkbordercolor)
    citecolor=cyan,        % color of links to bibliography
}
\newtheorem{theorem}{Theorem}[]
\newtheorem{lemma}{Lemma}[section]

\newtheorem{definition}{Definition}
\theoremstyle{remark}
\newtheorem{remark}[lemma]{Remark}
\newtheorem{example}{Example}
\def\ssum{\mathop{\sum\!\sum}}

\newcommand{\sumstar}{\sideset{}{^{*}}\sum}
\newcommand{\prodstar}{\sideset{}{^{*}}\prod}
\makeatletter

\newcommand{\Rmnum}[1]{\expandafter\@slowromancap\romannumeral #1@}
\makeatother

\def\P{\partial}

\def\li{{\rm Li}}

\def\qb{\mathbb Q}
\def\rb{\mathbb R}
\def\nb{\mathbb N}
\def\zb{\mathbb Z}
\def\cb{{\mathbb C}}
\def\sh{{\mathscr H}}
\def\cf{{\mathcal F}}
\def\rc{{\mathcal R}}
\def\cs{{\mathcal S}}
\def\ch{{\mathcal H}}
\def\cm{{\mathcal M}}
\def\ss{{\mathscr S}}

\def\rrw{\rightarrow}
\numberwithin{equation}{section}

\begin{document}

\title{Asymptotic formulae for Eulerian series}
\author{Nian Hong Zhou}
\address{Department of Mathematics,
East China Normal University,
500 Dongchuan Road, Shanghai 200241, PR China}
\email{nianhongzhou@outlook.com}
\thanks{}
\date{}

\dedicatory{}

\begin{abstract} Let $(a;q)_{\infty}$ be the $q$-Pochhammer symbol and $\li_2(x)$ be the dilogarithm function. Let $\prod_{\alpha,\beta,\gamma}$ be a finite product with every triple $(\alpha,\beta,\gamma)\in(\rb_{>0})^3$ and $S_{\alpha\beta\gamma}\in\rb$.  Also let the triple $(A,B,v)\in\left(\rb_{>0}\times\rb^2\right)\cup\left(\{0\}^2\times\rb_{>0}\right)\cup\left(\{0\}\times\rb_{<0}\times\rb\right)$. In this work, we let $z=e^v$, denote by $H_{-1}(u)=vu-Au^2+\sum_{\alpha}\li_2(e^{-\alpha u})\sum_{\beta,\gamma} \beta^{-1}S_{\alpha\beta\gamma}$ and consider the Eulerien series
\[\ch(z;q)=\sum_{m=0}^{\infty}\frac{q^{Am^2+Bm}z^{m}}{\prod\limits_{\alpha,\beta,\gamma}(q^{\alpha m+\gamma};q^{\beta})_{\infty}^{S_{\alpha\beta\gamma}}}.\]
We prove that if there exist an $\varepsilon>0$ such that $H_{-1}(u)$ is an increasing function on $[0,\varepsilon)$, then as $q\rrw 1^-$,
\[\ch(z;q)=\left(1+o\left(|\log q|^p\right)\right)\int\limits_{0}^{\infty}\frac{q^{Ax^2+Bx}z^{x}}{\prod\limits_{\alpha,\beta,\gamma}(q^{\alpha x+\gamma};q^{\beta})_{\infty}^{S_{\alpha\beta\gamma}}}\,dx\]
holds for each $p\ge 0$. We also obtain full asymptotic expansions for $\ch(z;q)$ which satisfy above condition as $q\rightarrow 1^{-}$. The complete asymptotic expansions for related basic hypergeometric series could be
derived as special cases.
\end{abstract}

\maketitle

\section{Introduction and statement of results}
We begin with the definition of Eulerian series, which could be found in \cite{MR2425181}.
\begin{definition}
Eulerian series are combinatorial formal power series which are constructed
from basic hypergeometric series.
\end{definition}

In his last letter to Hardy, Ramanujan listed $17$ examples of functions in Eulerian series that he called mock theta functions. The first three pages in which Ramanujan explained what he meant by
a "mock theta function" are very obscure. Hardy comments that a mock theta function is a function
defined by a $q$-series convergent when $|q|<1$, for which we can calculate
asymptotic formulae, when $q$ tends to a "rational point" $e^{2\pi is/r}$ of the unit
circle, of the same degree of precision as those furnished for the ordinary
theta functions by the theory of linear transformation (see \cite{MR1573993}).

In the same latter, Ramanujan also noted that for other Eulerian series, approximations analogous to mock theta function may not exist. He claimed that
\begin{equation*}
\sum_{m=0}^{\infty}\frac{q^{\frac{m(m+1)}{2}}}{(q;q)_m^2}=\sqrt{\frac{t}{2\pi\sqrt{5}}}\exp\left(\frac{\pi^2}{5t}+c_1t+\cdot\cdot\cdot+c_pt^p+O(t^{p+1})\right)
\end{equation*}
holds for each integer $p\ge 1$, $q=e^{-t}$, $t\rrw 0^+$ with infinitely many $c_j\neq 0$. Here we use the the $q$-Pochhammer symbol
$
(a;q)_{m}=\prod_{k=0}^{m-1}(1-aq^{k})
$
for $a, q\in\cb$, $|q|<1$ and $m\in\nb\cup \{\infty\}$.  Although this example have been discussed by Watson \cite{MR1573993} and McIntosh \cite{MR1310726},  we can't prove this claim until today. In same paper \cite{MR1310726}, McIntosh also provided the complete asymptotics: Let $a, b\in \rb_{>0}$, $c, t\in\rb$, $q=e^{-t}$ with $t\rrw 0^+$, then
\begin{equation}\label{mci1}
\sum_{m=0}^{\infty}\frac{q^{bm^2+cn}}{(q;q)_n}a^m\sim\exp\left(\frac{C_{-1}}{t}+\sum_{k=0}^{\infty}C_{k}t^k\right),
\end{equation}
 where $C_k$ are constants depends only on $a$, $b$ and $c$. He note that his method could applicable to wide variety unimodal series with some limit conditions which required, see Theorem 2 of \cite{MR1618298}.

 Moreover, let
$A$ be a positive definite symmetric $r\times r$ matrix, $B$ a vector of length $r$, and
$C$ a scalar, all three with rational coefficients. Zagier \cite[Section 3, Chapter II]{MR2290758}  define a function $f_{A,B,C}(z)$
by the $r$-fold $q$-hypergeometric series
\begin{equation}\label{2}
f_{A,B,C}(z)=\sum_{{\bf x}=(x_1,\dots,x_r)\in\nb^r}\frac{q^{\frac{1}{2}{\bf x}^TA{\bf x}+{\bf x}^TB+C}}{(q;q)_{x_1}\dots(q;q)_{x_r}},
\end{equation}
with $z\in\cb(\Im z>0)$ and ask when $f_{A,B,C}(z)$ is a modular function. Zagier give a method involving the asymptotic expansion of \eqref{2}, with $q=e^{-t}$ for $t\rrw 0^+$. In \cite{MR2290758}, Zagier also outline some methods to computing the asymptotics and solve the question for $r=1$. In \cite{MR2864462}, Vlasenko and Zwegers use the ideal comes from Zagier to give the asymptotics for $r\ge 2$.

Finally, the $q$-hypergeometric series are $q$-analogue generalizations of generalized hypergeometric series. In 1940, Wright \cite{MR0001296,MR0003876} has been established six theorems on the asymptotic expansion of the generalized hypergeometric function. Zhang \cite{MR2382736,MR2845514} has investigated Plancherel-Rotach type asymptotics for certain $q$-hypergeometric series. However, nothing of the asymptotics as \eqref{mci1} is known for more general $q$-hypergeometric series.\newline

The purposes of this paper is establish a complete asymptotic expansion for more general Eulerian series or $q$-hypergeometric series. We first fixed the following $q$-notation:
\[(a;q)_z=\frac{(a ;q)_{\infty}}{(aq^{z};q)_{\infty}},~ (a_{1},a_{2},\dotsc,a_{n};q)_{z}=\prod_{j=1}^{n}(a_{j};q)_{z}
\]
for $a, a_1,\dots a_n, q\in\cb$, $|q|<1$ and $z\in\cb$. We will focus on the formal Eulerian series
\begin{equation}\label{mqs}
\sh(z;q)=\sum_{m=0}^{\infty}\frac{q^{Am^2+Bm}z^m}{\prod_{a,b,c,d}\left(q^{a};q^{b}\right)_{c m+d}^{\cs(a,b,c,d)}},
\end{equation}
where $q\in(0,1)$, $z=e^v$; the triple $(A,v,B)\in\left(\rb_{>0}\times\rb^2\right)\cup\left(\{0\}^2\times\rb_{>0}\right)\cup\left(\{0\}\times\rb_{<0}\times\rb\right)$; the product $\prod_{a, b, c, d}$ is a finite product, every quadruple $(a,b,c,d)\in\rb^4$ with $b,c,a+bd>0$ and $\cs(a,b,c,d)\in\rb $ is a function in $a,b,c,d$.  Clearly,
\begin{equation}\label{520}
\sh(z; q)=\ch(z; q)\prod_{a,b,c,d}\left(q^{a};q^{b}\right)_{\infty}^{-\cs(a,b,c,d)},
\end{equation}
where
\begin{equation}\label{mqs0}
\ch(z;q)=\sum_{m=0}^{\infty}\left(\prod_{a,b,c,d}\left(q^{bcm+a+bd};q^{b}\right)_{\infty}^{\cs(a,b,c,d)}\right)q^{Am^2+Bm}z^m.
\end{equation}
The asymptotics of the general term of the product in \eqref{520} has been understood well by McIntosh \cite{MR1703273}. Thus we just need consider $\ch(z;q)$, which could be rewritten as the following simple form
\begin{equation}\label{mqs1}
\ch(z;q)=\sum_{m=0}^{\infty}\frac{q^{Am^2+Bm}z^m}{\prod_{\alpha,\beta,\gamma}(q^{\alpha m+\gamma};q^{\beta})_{\infty}^{S_{\alpha\beta\gamma}}}.
\end{equation}

In order to formulate the main result of this paper, we first denote by
\begin{align}\label{Hm1}
H_{-1}(u)&=u\log z-Au^2+\sum_{\alpha}\li_2(e^{-\alpha u})\sum_{\beta,\gamma} \beta^{-1}S_{\alpha\beta\gamma}\nonumber\\
&:=u\log z-Au^2-\sum_{1\le j\le H}\li_2(e^{-\alpha_j u})f(\alpha_j)
\end{align}
with $\li_2(\cdot)$ is the dilogarithm function be defined by \eqref{PLF1}, $f(\alpha_j)\neq 0$ for $1\le j\le H$ and $\alpha_1<\alpha_2<\dots<\alpha_H$. Comparing \eqref{mqs0} and \eqref{mqs1} we also have
\begin{equation}\label{Hm2}
H_{-1}(u)=vu-Au^2-\sum_{b,c}\li_{2}(e^{-bc u})\sum_{a,d}\frac{\cs(a,b,c,d)}{b}.
\end{equation}
Then, our main results as follows.
\begin{theorem}\label{mt}Let $\ch(z;q)$, $\sh(z;q)$ and $H_{-1}(u)$ be defined as above. If there exist an $\varepsilon>0$ such that $H_{-1}(u)$ is an increasing function on $[0,\varepsilon)$, then as $q\rrw 1^-$,
\[\ch(z;q)=\left(1+o\left(|\log q|^p\right)\right)\int\limits_{0}^{\infty}\frac{q^{Ax^2+Bx}z^{x}}{\prod_{\alpha,\beta,\gamma}(q^{\alpha x+\gamma};q^{\beta})_{\infty}^{S_{\alpha\beta\gamma}}}\,dx\]
holds for each $p\ge 0$. Furthermore, we have
\[\sh (z;q)=\left(1+o\left(|\log q|^p\right)\right)\int\limits_{0}^{\infty}\frac{q^{Ax^2+Bx}z^{x}}{\prod_{a,b,c,d}\left(q^a;q^b\right)_{c x+d}^{\cs(a,b,c,d)}}\,dx\]
holds for each $p\ge 0$ as $q\rrw 1^-$.
\end{theorem}
We shall first prove the following more general results and then Theorem \ref{mt} could be derived as special cases.
\begin{lemma}\label{th1}Let $F(x,t)$ and $E_{\ell}(u)(\ell\in\zb_{\ge -1})$ be some real analytic functions in $(0,\infty)$. For each $x \gg 1/t$ and $m\in\nb$, suppose that $F(x,t)$ satisfies the complete asymptotics
\[\frac{\P^mF(x,t)}{\P x^{m}}\sim t^{m}\sum_{\ell=-1}^{\infty}E_{\ell}^{\langle m\rangle}(xt)t^{\ell}\]
for $t\rrw 0^+$. Also suppose that the set of all local maximum points of $E_{-1}(u)$ on $(0,\infty)$ is a nonempty finite set $\ss_E$. Then, for each $p\ge 0$, as $t\rrw 0^+$
\begin{equation*}
\sum_{u_m/t< m\le u_M/t }\exp\left({F(m,t)}\right)= \left(1+o\left(t^p\right)\right)\int\limits_{u_m/t}^{u_M/t}\exp\left({F(x,t)}\right)\,dx,
\end{equation*}
where $u_m=\min\limits_{u\in\ss_E}u-1/|\log t|$ and $u_M=\max\limits_{u\in\ss_E}u+1/|\log t|$.
\end{lemma}
We have the complete asymptotic  expansion for above Lemma \ref{th1}
\begin{lemma}\label{th2}Let $F(x,t)$, $E_{\ell}(u)(\ell\in\zb_{\ge -1})$, $u_m$ and $u_M$ be defined as Lemma \ref{th1}.
Let $\kappa_{2\ell}(u,t)$ be defined by \eqref{kappa}. Then we roughly have for each $\ell\in\nb_1$,
\[\kappa_{2\ell}(u,t)\ll t^{\frac{\ell}{k_u(2k_u+1)}}.\]
Further more, we have the asymptotic expansion in $\kappa_{2\ell}(u,t)$ of the form
\begin{equation*}
\sum_{u_m/t<m\le u_M/t}e^{F(m, t)}\sim \sum_{
u\in\ss_E}\frac{\exp(F(u/t,t))}{V(u,t)}\sum_{\ell=0}^{\infty}\Gamma\left(\frac{2\ell+1}{2k_u}\right)\frac{\kappa_{2\ell}(u,t)}{k_u},
\end{equation*}
where $k_u$ is the minimum positive integer such that
$E_{-1}^{\langle 2k_u\rangle}(u)\neq 0$, $\kappa_{0}(u,t)=1$ and $V(u,t)$ be defined by \eqref{V}. In particular, we have the leading asymptotics
\[\sum_{u_m/t<m\le u_M/t}e^{F(m, t)}\sim \sum_{
u\in\ss_E}\frac{e^{E_0(u)}}{k_u}\Gamma\left(\frac{1}{2k_u}\right)\left(\frac{-(2k_u)!}{E_{-1}^{\langle 2k_u\rangle}(u)}\right)^{\frac{1}{2k_u}}t^{-1+\frac{1}{2k_u}}e^{{E_{-1}(u)}/{t}}.\]
\end{lemma}
\begin{remark}
From the assumption of Lemma \ref{th1}, \eqref{V} and \eqref{kappa} it is not difficult to see that $V(u,t)$ and $\kappa_{2\ell}(u,t)$ have asymptotic expansions in powers of $t^{1/k_u}$, this implies that the asymptotic expansion of Lemma \ref{th2} could be rewritten as an asymptotic expansion in powers of $t^{1/k_u}$ of the form
\[\sum_{u_m/t<m\le u_M/t}e^{F(m, t)}\sim \sum_{
u\in\ss_E}C_F(u)t^{-1+\frac{1}{2k_u}}e^{{E_{-1}(u)}/{t}}\left(1+\sum_{\ell= 1}^{\infty}C_{Fj}(u)t^{\ell/k_u}\right)\]
with
\[C_F(u)=\frac{e^{E_0(u)}}{k_u}\Gamma\left(\frac{1}{2k_u}\right)\left(\frac{-(2k_u)!}{E_{-1}^{\langle 2k_u\rangle}(u)}\right)^{\frac{1}{2k_u}}\]
and for each $j\in\nb_1$, $C_{Fj}(u)\in\rb$ depends only on $F(\cdot)$ and $u$.
\end{remark}
Finally, we obtain the full asymptotic behavior of Eulerian series \eqref{mqs} and \eqref{mqs1}.
\begin{theorem}\label{mc}Let $H_{-1}(u)$ be defined by \eqref{Hm1} and $H_{\ell}(u)(\ell\in\zb_{\ge 0})$ be defined by Lemma \ref{fasy}. Also let $\ss_H$ be the set of all local maximum points of $H_{-1}(u)$ on $(0,\infty)$. Then, under the assumption of Theorem \ref{mt} we have
\[\ch\left(z;q\right)=N_{\ch}(t)+I_{\ch}(t),\]
where $N_{\ch}(t)=0$ for $\ss_H$ is an empty set and if $\ss_H$ nonempty then
\[N_{\ch}(t)\sim
\sum_{
u\in\ss_H}C_{u}t^{-1+\frac{1}{2m_{u}}}e^{H_{-1}(u)/t}\bigg(1+\sum_{j\ge 1}C_{j}(u)t^{j/m_{u}}\bigg)
\]
with the coefficients $C_{j}(u)\in\rb$ are constant depends only on $\ch(z;\cdot)$ and $u$ be determined by \eqref{formula}, $m_{u}$ is the minimum positive integer such that $H_{-1}^{\langle 2m_u\rangle}(u)\neq 0$ and
\[
C_u=\frac{e^{H_0(u)}\Gamma\left(\frac{1}{2m_{u}}\right)}{m_{u}}\left(\frac{-(2m_{u})!}{H_{-1}^{\langle2m_{u}\rangle}(u)}\right)^{\frac{1}{2m_{u}}}
;\]
$I_{\ch}(t)=0$ for $A>0$ or $\log z<0$ or $f(\alpha_1)<0$ and if $f(\alpha_1)>0$ then
\[
I_{\ch}(t)\sim \frac{\Gamma\left({B}/{\alpha_1}\right)}{\alpha_1[f(\alpha_1)]^{B/\alpha_1}}t^{B/\alpha_1-1}\bigg(1+\sum_{\lambda\in\Lambda(\ch)}C_{\lambda}(\ch)t^{\lambda}\bigg)
,\]
where the set $\Lambda(\ch)\subset \qb_{>0}$ satisfy $\inf_{\lambda\neq \mu,\lambda,\mu\in\Lambda(\ch)}|\lambda-\mu|>0$ be defined by \eqref{Lambda}, and $C_{\lambda}(\ch)\in\rb$ for $\lambda\in\Lambda(\ch)$ be determined by \eqref{formula2} depends only on $\ch(1; \cdot)$.
\end{theorem}
\begin{remark}
We can obtain the complete asymptotic expansion for $\sh(z;q)$. For each $\ell\in\nb$, let the Bernoulli polynomials $B_{\ell}(x)$ and the Bernoulli number $B_{\ell}$ be defined by \eqref{BP} and \eqref{BN}, respectively. Then, we have
\begin{equation*}
\sh(z;q)\sim \ch(z;q)C_{\ch}t^{B_{\ch}}\exp\left(\frac{A_{\ch}}{t}+\sum_{\ell=1}^{\infty}A_{\ell}t^{\ell}\right)
\end{equation*}
with
\[A_{\ch}=\sum_{a,b,c,d}\frac{\pi^2\cs(a,b,c,d)}{6b}, A_{\ell}=\sum_{a,b,c,d}\frac{B_{\ell}\cs(a,b,c,d)b^{\ell}}{\ell(\ell+1)!}B_{\ell+1}\left(\frac{a}{b}\right)~ \ell\in\nb_1,\]
\[B_{\ch}=\sum_{a,b,c,d}\left(\frac{a}{b}-\frac{1}{2}\right)\cs(a,b,c,d)\;\mbox{and}\;C_{\ch}=\prod_{a,b,c,d}\left(\Gamma(a/b)/\sqrt{2\pi}\right)^{\cs(a,b,c,d)}.\]
\end{remark}

\begin{remark}It is clear that the asymptotic results of \cite{MR1310726} and \cite{MR1618298} of McIntosh could be derived as special cases of the Eulerian series \eqref{mqs} with the conditions which satisfied. Moreover, our theorem is more general, the asymptotic expansion is similar  with the definition of the mock theta function of Gordon and McIntosh \cite{MR1783627}. Furthermore, some special $\sh(1;q)$ are generating
functions in many partition problems, for example, our results can give the complete asymptotic expansion for all partition generating functions in Bringmann and Mahlburg \cite{MR3213313}. Using the Tauberian theorem of Ingham  (see \cite[Theorem 3.1]{MR3213313} and \cite{MR0005522}), which allows us to describe the asymptotic behavior of the coefficients of a power series using the analytic nature of the partition generating functions $\sh(1;q)$.
\end{remark}

This paper is organized as follows.
In Section \ref{sec2}, we first determine the stationary point of $F(x,t)$ and consider it's Taylor expansion at the stationary point.
In Section \ref{sec3}, we prove  Lemma \ref{th1} and  Lemma \ref{th2}.
In Section \ref{sec4}, we first collect properties on Bernoulli numbers, Bernoulli polynomials and the polylogarithm function. Then, we give the asymptotic facts for the general term of the product in \eqref{520} and prove that the logarithm of general term of $\ch(z;q)$ satisfies the assumption of $F(m,t)$ in  Lemma \ref{th1}.
In Section \ref{sec5} we prove Theorem \ref{mt} and Theorem \ref{mc}. We will prove Theorem \ref{mt} in Subsection \ref{sec51}--\ref{sec53} and prove Theorem \ref{mc} in last subsection of this section.
In Section \ref{sec6}, we apply our main theorem to the some confluent basic hypergeometric series, some simple Eulerien series and some mock theta functions.

\section{The stationary point of $F(x,t)$}\label{sec2}

\begin{lemma}\label{mainlemma}Let $t>0$ sufficiently small be fixed, let $F(x,t), E_{\ell}(u)$
and $\ss_E$ be defined as Lemma \ref{th1}.
Also let $\ss_{F}$ be the set of all maximum point of $F(x,t)$ for all $x\asymp 1/t$. Then for each $X\in\ss_{F}$ there exist an $u\in\ss_E$ such that
\[X=\left(u+r_u(t)\right)/t,\]
with $r_u(t)\ll t^{1/(2k_u-1)}$ and where $k_u$ is the minimum positive integer such that $E_{-1}^{\langle 2k_u\rangle}(u)\neq 0$.
\end{lemma}
\begin{proof} Letting $X\in\ss_{F}$ then
\begin{equation}\label{hd1}
0=\frac{\P F}{\P x}(X,t)\sim E_{-1}'(Xt)+\sum_{\ell=1}^{\infty}E_{\ell-1}'(Xt)t^{\ell}.
\end{equation}
Thus we have $\lim_{t\rrw 0^{+}}(Xt)\in \ss_E$, namely, $X\sim u/t$ for some $u\in\ss_E$. On the other hand, it is clear that for each $u\in \ss_E$ there exist an $r_u(t)\in\rb$ with $r_u(t)=o(t)$ for $t\rrw 0^+$ such that $X=(u+r_u(t))/t$ satisfies \eqref{hd1}. Further more, the using of Taylor theorem yields
\[\sum_{k=1}^{\infty}\frac{E_{-1}^{\langle k+1 \rangle}(u)}{k!}r_u(t)^k+\sum_{\ell=1}^{\infty}t^{\ell}\sum_{k=0}^{\infty}\frac{E_{\ell-1}^{\langle k+1 \rangle}(u)}{k!}r_u(t)^k\sim 0.\]
Since $u$ is the maximum point of $E_{-1}(x)$ in $x\in(0,\infty)$, so the minimum positive integer $k$ such that $E_{-1}^{\langle k+1\rangle}(u)\neq 0$ must be an odd integer $2k_u-1$. Moreover, we have $E_{-1}^{\langle2k_{u}\rangle}(u)<0$ and
\begin{equation*}
\sum_{k=2k_u-1}^{\infty}\frac{E_{-1}^{\langle k+1 \rangle}(u)}{k!}r_u(t)^k+\sum_{\ell=1}^{\infty}t^{\ell}\sum_{k=0}^{\infty}\frac{E_{\ell-1}^{\langle k+1 \rangle}(u)}{k!}r_u(t)^k\sim 0.
\end{equation*}
This implies that $r_u(t)\ll t^{1/(2k_u-1)}$. Which completes the proof of the lemma.
\end{proof}

Next, we obtain the Taylor series for $F(m,t)$ at $u/t$ with $u\in\ss_E$.
\begin{lemma}\label{taylor} Let $F(x,t), E_{\ell}(u)$ and $\ss_E$ be defined as Lemma \ref{th1}. Also let $t>0$ sufficiently small. Then there exist a constant $\theta_{E}>0$ depends only on $E_{-1}(\cdot)$ such that for $m\in\left[(u-\theta_E)/t,(u+\theta_E)/t\right]$ we have the Taylor expansion
\[
F(m,t)=\sum_{\ell=0}^{\infty}\frac{1}{\ell !}\frac{\P^{\ell}F}{\P x^{\ell}}(u/t,t)(m-u/t)^{\ell}.
\]
\end{lemma}
\begin{proof}
First, we have for each $N\in\nb$,
\[
{\P^N F(x,t)}/{\P x^N}=O\left(t^{N-1}\left|E_{-1}^{\langle N\rangle}(xt)\right|\right)
\]
holds for each $x\gg 1/t$. Thus for $m\in\left[(u-\theta)/t,(u+\theta)/t\right]$ with $\theta\in (0,u/2)$, we find that
\begin{equation*}
\int\limits_{u/t}^{m}\frac{\P^{N+1} F(x,t)}{\P x^{N+1}}\frac{(m-x)^N}{N!}{\rm d}x\ll \frac{t^{N}\left|m-u/t\right|^{N+1}}{(N+1)!}\max_{|x-u/t|\le \theta/t}\left|E_{-1}^{\langle N+1\rangle}(xt)\right|.
\end{equation*}
We note that $E_{-1}(u)$ is real analytic function, then there exist a $C_u>0$ depends only on $u$ and $E_{-1}(\cdot)$ such that
\[\sup_{0<\theta< u/2}\left(\frac{1}{(N+1)!}\max_{|x-u/t|\le \theta/t}\left|E_{-1}^{\langle N+1\rangle}(xt)\right|\right)\le C_u^{N+1}.\]
Thus
\[
\int\limits_{u/t}^{m}\frac{\P^{N+1} F(x,t)}{\P x^{N+1}}\frac{(m-x)^N}{N!}{\rm d}x\ll \frac{C_{u}^{N+1}}{t}\theta^{N+1}.
\]
Then, by setting
\[\theta_E=\min_{u\in \ss_E}\left(\frac{u}{3+2uC_u}\right)\]
and recalling the Taylor theorem
\[
F(m,t)=\sum_{\ell=0}^{N}\frac{\P^{\ell}F}{\P x^{\ell}}(u/t,t)\frac{(m-u/t)^{\ell}}{\ell !}+\int\limits_{u/t}^{m}\frac{\P^{N+1} F(x,t)}{\P x^{N+1}}\frac{(m-x)^N}{N!}{\rm d}x,
\]
we immediately obtain the proof of this lemma.
\end{proof}

\section{The proof of the Lemma \ref{th1} and Lemma \ref{th2}}\label{sec3}
\subsection{The proof of the Lemma \ref{th1}}
Let $t>0$ sufficiently small be fixed. We have first
\begin{align}\label{mm1}
\sum_{u_m/t<m\le u_M/t}e^{F(m, t)}&=\sum_{m\in \cm_{F}}\exp\left[F(m,t)\right]+\sum_{\substack{u_m/t<m\le u_M/t\\ m\not\in\cm_{F}}}\exp\left[F(m,t)\right]\nonumber\\
&:=M_{F}+R_{F},
\end{align}
where the number set
\[\cm_{F}=\bigcup_{u\in\ss_{E}}\cm_{u}\;\mbox{with}\;\cm_{u}=\{m\in\nb: |m-u/t|\le t^{\theta_u-1}\}\]
and $\theta_u\in (1/(2k_u+1), 1/(2k_u))$ for each $u\in\ss_E$, where $k_u$ be defined as Lemma \ref{mainlemma}. It is easily seen that the above is a disjoint union as $t\rrw 0^+$.

\subsubsection{The estimate of $R_F$}
For $u\in\ss_E$ and $m\in \cm_{u}$ we note that
\begin{equation}
\frac{\P^{\ell}F}{\P x^{\ell}}(u/t,t)\sim t^{\ell-1}\left(E_{-1}^{\langle \ell\rangle}(u)+tE_{0}^{\langle \ell\rangle}(u)\right)\ll \begin{cases}t^{\ell}\quad &\ell\in[1, 2k_u-1]\\
 t^{\ell-1}&\ell\ge 2k_u
\end{cases}
\end{equation}
holds for each $\ell\in\nb_1$, thus the using of Lemma \ref{taylor} yields
\begin{align}\label{mm2}
F(m,t)&=\sum_{\ell\ge 0}\frac{1}{\ell !}\frac{\P^{\ell}F}{\P x^{\ell}}(u/t,t)(m-u/t)^{\ell}\nonumber\\
&=F(u/t,t)+\frac{1}{(2k_u)!}\frac{\P^{2k_u}F}{\P x^{2k_u}}(u/t,t)(m-u/t)^{\ell}+o(1)\nonumber\\
&=F(u/t,t)+\frac{1}{(2k_u)!}E_{-1}^{\langle 2k_u\rangle}(u)(m-u/t)^{2k_u}t^{2k_u-1}+o(1)
\end{align}
holds for $t\rrw 0^+$. Therefore from the facts that $E_{-1}^{\langle 2k_u\rangle}(u)<0$ for each $u\in\ss_E$, the monotonicity of $F(m,t)$ and \eqref{mm2} we obtain that
\begin{align}\label{mm4}
R_{F}&\ll \frac{1}{t}\sum_{u\in\ss_{E}}e^{F(u/t,t)}\exp\left(\frac{E_{-1}^{\langle 2k_u\rangle}(u)}{(2k_u)!}t^{-1+2\theta_uk_u}+o(1)\right)\nonumber\\
&\ll \sum_{u\in\ss_{E}}\exp\left(F\left(u/t,t\right)-\delta_ut^{-1+2\theta_uk_u}\right)
\end{align}
for some $\delta_u>0$ depends only on $E_{-1}(\cdot)$ and $u$.

\subsubsection{The estimate of $M_F$}
We write
\begin{align*}
T_{F}(u,t)&=\sum_{|m-u/t|\le t^{\theta_u-1}}\exp\left[F(m,t)-F\left(u/t,t\right)\right]\\
&=\sum_{|m-u/t|\le t^{\theta_u-1}}\exp\left[\sum_{\ell\ge 1}\frac{1}{\ell !}\frac{\P^{\ell}F}{\P x^{\ell}}(u/t,t)(m-u/t)^{\ell}\right].
\end{align*}
We denote by
\[f(m,t)=\sum_{\ell\ge 1}\frac{1}{\ell !}\frac{\P^{\ell}F}{\P x^{\ell}}(u/t,t)(m-u/t)^{\ell},\]
then for $x\in \cm_u$ and each $N\in\nb_1$, it is not hard to prove that
\begin{equation}\label{hoe}
\P^N f(x,t)/{\P x^N}\ll_{N}\begin{cases}
t^{N-1+\theta_u} \quad &N\le 2k_u-1\\
t^{N-1}&N\ge 2k_u
\end{cases}\ll t^{\theta_uN}
\end{equation}
by Lemma \ref{taylor}. The Euler--Maclaurin formula (see for example \cite[Theorem D.2.1]{MR1688958}) gives that for each $N\in\nb$,
\begin{align}\label{EMF}
T_F(u,t)=&\int\limits_{\lceil(u-t^{\theta_u})/t\rceil}^{\lfloor(t^{\theta_u}+u)/t\rfloor}\left(e^{f(x,t)}-\frac{(-1)^{N}}{N!}B_N(x-\lfloor x\rfloor)\frac{\P^N e^{f(x,t)}}{\P x^N}\right)\,d x\nonumber\\
&+\sum_{\ell=0}^{N}\frac{(-1)^{\ell+1}B_{\ell+1}}{\ell+1}\frac{\P^{\ell} e^{f(x,t)}}{\P x^{\ell}}\bigg|_{\lceil(u-t^{\theta_u})/t\rceil}^{\lfloor(t^{\theta_u}+u)/t\rfloor}
\end{align}
where $\lfloor \cdot\rfloor$ and $\lceil \cdot\rceil$ are the greatest integer function and the least integer function, respectively. Using \eqref{mm2} it is not hard to see that
\begin{equation}\label{0be}
e^{f(x,t)}\bigg|_{\lceil(u-t^{\theta_u})/t\rceil}^{\lfloor(t^{\theta_u}+u)/t\rfloor}\ll \exp\left(\frac{E_{-1}^{\langle 2k_u\rangle}(u)}{(2k_u)!}t^{-1+2\theta_uk_u}+o(1)\right)\ll e^{-\delta_ut^{-1+2\theta_uk_u}},
\end{equation}
where $\delta_u$ be defined as \eqref{mm4}. From Fa\`{a} di Bruno's formula
\[
\frac{\P^N e^{f(x,t)}}{\P x^N}=e^{f(x,t)}\sum_{\substack{m_1,m_2,\dots m_N\in\nb\\m_1+2m_2+\dots+Nm_N=N}}\frac{N!}{\prod_{j=1}^Nm_j!j!^{m_j}} \prod_{j=1}^{N}\left(\frac{\P^j f(x,t)}{\P x^j}\right)^{m_j}
\]
and the estimate \eqref{hoe}, we obtain that
\begin{align}\label{Nbe}
&\frac{\P^{N} e^{f(x,t)}}{\P x^{N}}\bigg|_{\lceil(u-t^{\theta_u})/t\rceil}^{\lfloor(t^{\theta_u}+u)/t\rfloor}\nonumber\\
&\qquad\ll e^{-\delta_ut^{2\theta_uk_u-1}}\sum_{\substack{m_1,m_2,\dots m_N\in\nb\\m_1+2m_2+\dots+Nm_N=N}}\prod_{j=1}^{N}\left(t^{j\theta_u}\right)^{m_j}\ll e^{-\delta_ut^{2\theta_uk_u-1}}
\end{align}
for each $N\in\nb_1$, where $\delta_u$ also be defined as \eqref{mm4}. Further more,
\begin{align}\label{EFe}
\int\limits_{\lceil(u-t^{\theta_u})/t\rceil}^{\lfloor(t^{\theta_u}+u)/t\rfloor}\left|\frac{\P^N e^{f(x,t)}}{\P x^N}\right|\,d x
&\ll \int\limits_{|x-u/t|\le t^{\theta_u-1}}\,d xe^{f(x,t)}\sum_{\substack{m_1,m_2,\dots m_N\in\nb\\m_1+2m_2+\dots+Nm_N=N}}\prod_{j=1}^{N}\left(t^{j\theta_u}\right)^{m_j}\nonumber\\
&\ll t^{\theta_uN}\int\limits_{|x-u/t|\le t^{\theta_u-1}}e^{f(x,t)}\,d x
\end{align}
holds for each $N\in\nb_1$. On the other hand, by \eqref{mm2} we see
\begin{equation}\label{lb}
\int\limits_{|x-u/t|\le t^{\theta_u-1}}e^{f(x,t)}\,dx\gg \int\limits_{|y|\le t^{\theta_u-1}}\exp\left(\frac{y^{2k_u}t^{2k_u-1}}{(2k_u)!}E_{-1}^{\langle 2k_u\rangle}(u)\right)\,dy\gg 1/t.
\end{equation}
Thus from \eqref{EMF}--\eqref{lb}, it is clear that
\begin{equation}\label{T}
T_{F}(u,t)=\left(1+o(t^p)\right)\int\limits_{|x-u/t|\le t^{\theta_u-1}}\exp\left(F(x,t)-F(u/t,t)\right)\,d x,
\end{equation}
holds for each $p\ge 0$ as $t\rrw0^+$.
\subsubsection{The final estimate}
From \eqref{mm4}, \eqref{lb} and \eqref{T}, it is obvious that for each $p\ge 0$,
\begin{align}\label{imp1}
\sum_{u_m/t<m\le u_M/t}e^{F(m, t)}&=\sum_{u\in\ss_E}\left(\left(1+o(t^p)\right)\int\limits_{|x-u/t|\le t^{\theta_u-1}}e^{F(x,t)}\,d x+O\left(e^{F\left(u/t,t\right)-\delta_ut^{-1+2\theta_uk_u}}\right)\right)\nonumber\\
&=\left(1+o(t^p)\right)\sum_{u\in\ss_E}\int\limits_{|x-u/t|\le t^{\theta_u-1}}e^{F(x,t)}\,d x
\end{align}
holds for $t\rrw 0^+$. Further more, if we denote by
\[\mathfrak{m}=(u_m/t,u_M/t]\setminus \left(\cup_{u\in\ss_E}\{x\in\rb :|x-u/t|\le t^{\theta_u-1}\}\right).\]
Then by the monotonicity of $F(m,t)$, we obtain
\begin{equation}\label{minor}
\int\limits_{\mathfrak{m}}e^{F(x,t)}\,dx\ll \frac{u_M-u_m}{t}\sum_{u\in\ss_E}e^{F(u/t,t)-\delta_ut^{-1+2\theta_uk_u}}.
\end{equation}
Combining \eqref{imp1}, \eqref{minor} and similar with \eqref{imp1}, we obtain that for each $p\ge 0$, as $t\rrw 0^+$
\[\sum_{u_m/t<m\le u_M/t}e^{F(m,t)}=\left(1+o(t^p)\right)\int\limits_{u_m/t}^{u_M/t}e^{F(x,t)}\,d x.\]
Which completes the proof of the Lemma \ref{th1}.
\subsection{The proof of the Lemma \ref{th2}}\label{pt12} It is not difficult to compute that the integral in \eqref{T} equals to
\begin{align*}
&\int\limits_{|x|\le t^{\theta_u-1}}\exp\left(\sum_{\ell\ge 1}\frac{1}{\ell !}\frac{\P^{\ell}F}{\P x^{\ell}}(u/t,t)x^{\ell}\right)\,dx\\
&\qquad=\frac{1}{V(u,t)}\int\limits_{|y|\le t^{\theta_u-1}V(u,t)}\exp\left(-y^{2k_u}+\sumstar_{k\ge 1}\lambda_{k}(u,t)y^{k}\right)\,dy\\
&\qquad=\frac{1}{V(u,t)}\int\limits_{|y|\le t^{\theta_u-1/(2k_u)}}e^{-y^{2k_u}}\exp\left(\sumstar_{k\ge 1}\lambda_{k}(u,t)y^{k}\right)\,dy+O\left(e^{-\delta_ut^{-1+2\theta_uk_u}}\right),
\end{align*}
where $*$ means that $k\neq 2k_u$ and
\begin{equation}\label{V}
V(u,t)=\left(\frac{-1}{(2k_u) !}\frac{\P^{2k_u}F}{\P x^{2k_u}}(u/t,t)\right)^{\frac{1}{2k_u}}
\end{equation}
and
\begin{equation*}
\lambda_{\ell}(u,t)=\frac{1}{\ell !V(u,t)^{\ell}}\frac{\P^{\ell}F}{\P x^{\ell}}(u/t,t),\ell \in(\nb\setminus\{2k_u\}).
\end{equation*}
Further more, we have the estimates
\begin{align*}
\frac{1}{V(u,t)}=\left(-\frac{(2k_u)!}{E_{-1}^{\langle 2k_u\rangle}(u)}\right)^{\frac{1}{2k_u}}t^{1/(2k_u)-1}
\left(1-\frac{1}{2k_u}\frac{E_{0}^{\langle 2k_u\rangle}(u)}{E_{-1}^{\langle 2k_u\rangle}(u)}t+O(t^2)\right)
\end{align*}
and
\begin{align}\label{lambdae}
\lambda_{\ell}(u,t)\sim\left(-\frac{(2k_u)!}{E_{-1}^{\langle 2k_u\rangle}(u)}\right)^{\frac{\ell}{2k_u}}t^{\ell/(2k_u)}\times
\begin{cases}E_{0}^{\langle \ell\rangle}(u)\quad &\ell\in[1,2k_u)\\
t^{-1}E_{-1}^{\langle \ell\rangle}(u)\quad &\ell\in(2k_u,\infty).
\end{cases}
\end{align}
by Newton's generalized binomial theorem. If we write
\[
\exp\left[\sumstar_{k\ge 1}\lambda_{k}(u,t)x^{k}\right]=\sum_{\ell=0}^{\infty}\kappa_{\ell}(u,t)x^{\ell},
\]
then
\begin{align}\label{kappa}
\kappa_{\ell}(u,t)&=\sum_{\substack{\ell_r\in\nb\\ \sum_{r\ge 1}^*r\ell_r=\ell}}\prod_{\substack{r\ge 1\\ r\neq 2k_u}}\frac{[\lambda_{r}(u,t)]^{\ell_r}}{\ell_r !}\nonumber\\
&=\sum_{\substack{\ell_r\in\nb\\ \sum_{r\ge 1}^*r\ell_r=\ell}}\prod_{\substack{r\ge 1\\ r\neq 2k_u}}\frac{1}{\ell_r!}\left(\frac{[(2k_u)!]^{r/(2k_u)}}{r !}\frac{\frac{\P^{r}F}{\P x^{r}}(u/t,t)}{\left|\frac{\P^{2k_u}F}{\P x^{2k_u}}(u/t,t)\right|^{r/(2k_u)}}\right)^{\ell_r}.
\end{align}
Hence the using of \eqref{lambdae} and \eqref{kappa} implies that
\begin{align*}
\kappa_{\ell}(u,t)&\ll \sum_{\substack{\ell_r\in\nb\\ \sum_{r\ge 1}^*r\ell_r=\ell}}\prod_{r=1}^{2k_u-1}t^{\frac{r\ell_r}{2k_u}}\prod_{r=2k_u+1}^{\ell}t^{\frac{r\ell_r}{2k_u}-\ell_r}\\
&\ll t^{\frac{\ell}{2k_u}}\sum_{\substack{\ell_r\in\nb\\ \sum_{r\ge 1}^*r\ell_r=\ell}}\prod_{r=2k_u+1}^{\ell}t^{-\frac{r\ell_r}{2k_u+1}}\ll t^{\frac{\ell}{2k_u(2k_u+1)}}
\end{align*}
for each $\ell\in\nb$. Thus it is not hard to show that
\begin{align*}
T_F(u,t)&=\frac{1+o(t^N)}{V(u,t)}\left(\sum_{\ell=0}^{2k_u(2k_u+1)N}\kappa_{\ell}(u,t)\int\limits_{\rb}e^{-y^{2k_u}}y^{\ell}\,dy+o\left(t^{N}\right)\right)\\
&=\frac{1}{V(u,t)}\left(\sum_{\ell=0}^{k_u(2k_u+1)N}\Gamma\left(\frac{2\ell+1}{2k_u}\right)\frac{\kappa_{2\ell}(u,t)}{k_u}+o\left(t^{N}\right)\right)
\end{align*}
holds for each $N\in\nb$. Then combining \eqref{imp1} we finish the proof of Lemma \ref{th2}.
\section{Preliminary results of $\sh(z;q)$}\label{sec4}
To prove Theorem \ref{mt}, we first need the following concepts and lemmas.
\subsection{some special functions}
Recall that the Bernoulli polynomials $B_{\ell}(x)$ are involved in the generating function is
\begin{equation}\label{BP}
\frac{ze^{zx}}{e^{z}-1}=\sum_{\ell=0}^{\infty}\frac{z^{\ell}}{\ell !}B_{\ell}(x),
\end{equation}
where $|z|<2\pi$. The Bernoulli numbers are given by
\begin{equation}\label{BN}
B_{\ell}=B_{\ell}(0)~\mbox{for}~\ell\in\nb.
\end{equation}
It is a well known fact that
\begin{equation}\label{BN1}
B_0=1, B_1=-\frac{1}{2}, B_{2\ell+1}=0~\mbox{and}~B_{2\ell}\sim (-1)^{\ell+1}\frac{2(2\ell)!}{(2\pi)^{2\ell}}~\mbox{as}~\ell\rrw\infty~ \mbox{for}~\ell\in\nb_1.
\end{equation}

The polylogarithm function $\li_{s}(z)$ is
\begin{equation}\label{PLF}
\li_{s}(z)=\sum_{k=1}^{\infty}\frac{z^k}{k^{s}}.
\end{equation}
If $s=2$, then it is called that dilogarithm function, say
\begin{equation}\label{PLF1}
\li_{2}(z)=\sum_{k=1}^{\infty}\frac{z^k}{k^2}
\end{equation}
and converges for $|z|\le 1$.
If $s=1$, then
\begin{equation*}
\li_{1}(z)=\sum_{k=1}^{\infty}\frac{z^k}{k}=-\log(1-z),
\end{equation*}
which converges for $|z|<1$. If $s$ is a non-positive integer, say $s=-r$, then the polylogarithm function
is defined recursively by
\begin{equation}\label{PLF1}
 \li_{-r}(z)=z\frac{\,d}{\,d z}\li_{1-r}(z) ~\mbox{for}~r\in\nb
\end{equation}
and $|z|<1$. Thus we have the following lemma.
\begin{lemma}\label{lm2}(See \cite[Lemma 2]{MR1703273})Let integer $n\le 2$, $a$ and $x$ be real number. Then, we have
\[
\li_{n}(ae^x)=\sum_{k=0}^{\infty}\frac{\li_{n-k}(a)}{k!}x^k
\]
holds for $|a|< 1$ and $|x|<\min(-\log |a|,\pi)$.
\end{lemma}
We need the following asymptotic of $q$-shifted factorials which proof follows immediately from Theorem 2 of \cite{MR1703273} or Theorem 2 of \cite{MR3128410} with change of variables.
\begin{lemma}\label{mtct}Let $b>0$ and real $a/b\not\in\zb_{\le 0}$.  Then, we have
\begin{equation*}
(e^{-at};e^{-bt})_{\infty}\sim\frac{\sqrt{2\pi}t^{\frac{1}{2}-\frac{a}{b}}}{\Gamma(a/b)}\exp\left(-\frac{\pi^2}{6bt}-\sum_{\ell=1}^{\infty}\frac{b^{\ell}B_{\ell}B_{\ell+1}(a/b)}{\ell(\ell +1)!}t^{\ell}\right).
\end{equation*}
\end{lemma}
\subsection{Some asymptotic properties of $\sh(z;q)$}
\subsubsection{The asymptotics for the product term of $\sh(z;q)$}\label{sec41}
From Lemma \ref{mtct}, we have
\begin{align*}
&\prod_{a,b,c,d}(q^a;q^b)_{\infty}^{-\cs(a,b,c,d)}\\
&\qquad\sim\prod_{a,b,c,d}\left(\frac{\Gamma(a/b)t^{\frac{a}{b}-\frac{1}{2}}}{\sqrt{2\pi}}\exp\left(\frac{\pi^2}{6bt}+\sum_{\ell=1}^{\infty}\frac{b^{\ell}B_{\ell}B_{\ell+1}(a/b)}{\ell(\ell +1)!}t^{\ell}\right)\right)^{\cs(a,b,c,d)}\\
&\qquad=t^{\sum_{a,b,c,d}\left(\frac{a}{b}-\frac{1}{2}\right)\cs(a,b,c,d)}\prod_{a,b,c,d}\left(\Gamma(a/b)/\sqrt{2\pi}\right)^{\cs(a,b,c,d)}\\
&~~~\qquad\times\exp\left(\frac{1}{t}\sum_{a,b,c,d}\frac{\pi^2\cs(a,b,c,d)}{6b}+\sum_{\ell=1}^{\infty}t^{\ell}\sum_{a,b,c,d}\frac{B_{\ell}\cs(a,b,c,d)b^{\ell}}{\ell(\ell+1)!}B_{\ell+1}\left(\frac{a}{b}\right)\right).
\end{align*}
Namely, we have
\begin{equation*}
\prod_{a,b,c,d}(q^a;q^b)_{\infty}^{-\cs(a,b,c,d)}\sim C_{\ch}t^{B_{\ch}}\exp\left(\frac{A_{\ch}}{t}+\sum_{\ell=1}^{\infty}A_{\ell}t^{\ell}\right)
\end{equation*}
with
\[A_{\ch}=\sum_{a,b,c,d}\frac{\pi^2\cs(a,b,c,d)}{6b}, A_{\ell}=\sum_{a,b,c,d}\frac{B_{\ell}\cs(a,b,c,d)b^{\ell}}{\ell(\ell+1)!}B_{\ell+1}\left(\frac{a}{b}\right)~ \ell\in\nb_1,\]
\[B_{\ch}=\sum_{a,b,c,d}\left(\frac{a}{b}-\frac{1}{2}\right)\cs(a,b,c,d)\;\mbox{and}\;C_{\ch}=\prod_{a,b,c,d}\left(\Gamma(a/b)/\sqrt{2\pi}\right)^{\cs(a,b,c,d)}.\]
\subsubsection{Some asymptotic properties of $\ch(z;q)$}
We first denote by
\begin{equation*}
\cf(m,t)=\log\left(\left(\prod_{a,b,c,d}\left(q^{bcm+a+bd};q^{b}\right)_{\infty}^{\cs(a,b,c,d)}\right)q^{Am^2+Bm}z^m\right),
\end{equation*}
then
\begin{equation}\label{sum}
\ch(z;q)=\sum_{m\in\nb}\exp\left(\cf(m,t)\right).
\end{equation}
It is not hard to compute that
\begin{align}\label{cf}
\cf(m,t)&=\sum_{a,b,c,d}\cs(a,b,c,d)\sum_{\ell=0}^{\infty}\log\left(1-q^{bcm+a+bd+b\ell}\right)-Am^2t-Bmt+mv\nonumber\\
&=mv-Am^2t-Bmt-\sum_{a,b,c,d}\cs(a,b,c,d)\sum_{\ell=0}^{\infty}\sum_{k=1}^{\infty}\frac{q^{k(bcm+a+bd)+kb\ell}}{k}\nonumber\\
&=mv-Am^2t-Bmt-\sum_{a,b,c,d}\cs(a,b,c,d)\sum_{k=1}^{\infty}\frac{e^{-k(bcm+a+bd)t}}{k(1-e^{-kb t})}
\end{align}
or, we have
\begin{equation}\label{cf1}
\cf(m,t)=mv-Am^2t-Bmt+\sum_{\alpha,\beta,\gamma}S_{\alpha\beta\gamma}\sum_{k=1}^{\infty}\frac{e^{-k(\alpha m+\gamma)t}}{k(1-e^{-k\beta t})}
\end{equation}
by \eqref{mqs1}.
The series $\cf(m,t)$ converges uniformly on compact subsets of $t>0$ and $m\ge 0$. Moreover, it is easily seen that the function
$\cf(x, t)$ with $(x,t)\in (\rb_{>0})^{2}$ is a real analytic function. Under the following lemma, we can apply Lemma \ref{th1} to prove Theorem \ref{mt}.
\begin{lemma}\label{fasy} Let $\cf(x,t)$ be defined as \eqref{cf} and let the real number $\delta>0$. Then, for all $X\ge t^{-\delta}$, we have
\[
\frac{\P^N\cf(X,t)}{\P X^{N}}= t^N\left(\sum_{\ell=-1}^{M}H_{\ell}^{\langle N\rangle}(Xt)t^{\ell}+o(t^{\delta M/2})\right)
\]
as $t\rrw 0^+$ for each $N, M\in \nb$, where $H_{-1}(u)$ be defined by \eqref{Hm1} and
\begin{equation*}
H_{m}(u)=\begin{cases}
\sum\limits_{a,b,c,d}\li_{1}(e^{-bc u})\left(d+\frac{a}{b}-\frac{1}{2}\right)\cs(a,b,c,d)-Bu &m=0\\
(-1)^{m}\sum\limits_{a,b,c,d}\frac{b^m\cs(a,b,c,d)}{(m+1)!}B_{m+1}\left(\frac{a+bd}{b}\right)\li_{1-m}(e^{-bc u})&m\in\nb_1,
\end{cases}
\end{equation*}
or
\begin{equation*}
H_{m}(u)=\begin{cases}
-\sum\limits_{\alpha,\beta,\gamma}\li_{1}(e^{-\alpha u})\left(\frac{\gamma}{\beta}-\frac{1}{2}\right)S_{\alpha\beta\gamma}-Bu &m=0\\
(-1)^{m-1}\sum\limits_{\alpha,\beta,\gamma}\frac{\beta^mS_{\alpha\beta\gamma}}{(m+1)!}B_{m+1}\left(\gamma/\beta\right)\li_{1-m}(e^{-\alpha u})&m\in\nb_1
\end{cases}
\end{equation*}
by comparing \eqref{cf} and \eqref{cf1}. In particular, we have for each $X\gg 1/t$
\begin{equation}\label{fasy1}
\frac{\P^N\cf(X,t)}{\P X^{N}}= t^N\left(\sum_{\ell=-1}^{M}H_{\ell}^{\langle N\rangle}(Xt)t^{\ell}+O(t^{ M+1})\right)
\end{equation}
as $t\rrw 0^+$ for each $N\in \nb$.
\end{lemma}
\begin{proof} We shall prove the case $N=0$, the proof of the cases $N\ge 1$ is similar. It is not hard to see that
\begin{align*}
\cf(X,t)&=Xv-AX^2t-BXt\\
&-\sum_{a,b,c,d}\frac{\cs(a,b,c,d)}{bt}\sum_{k\le t^{2\delta/3-1}}\frac{e^{-k(bcX+a+bd)t}}{k^2}\frac{-kb t}{e^{-kb t}-1}+R_1(X,t),
\end{align*}
where
\[
R_1(X,t)\ll \sum_{a,b,c,d}\sum_{k> t^{2\delta/3-1}}\frac{e^{-kbcXt}}{kt}\ll e^{-t^{-\delta/4}}
\]
by using the condition $X\ge t^{-\delta}$. By \eqref{BN} we have
\[
\sum_{k\le t^{2\delta/3-1}}\frac{e^{-k(bcX+a+bd)t}}{k^2}\frac{-kb t}{e^{-kb t}-1}=\sum_{k\le t^{2\delta/3-1}}\frac{e^{-k(bcX+a+bd)t}}{k^2}\sum_{\ell=0}^{\infty}\frac{B_{\ell}}{\ell !}(-kb t)^{\ell}
\]
Then use the asymptotic formula of $B_{\ell}$ (see \eqref{BN1}), for each $M\in\nb$
\begin{equation*}
\sum_{k\le t^{2\delta/3-1}}\frac{e^{-k(bcX+a+bd)t}}{k^{2}}\frac{-kb t}{e^{-kb t}-1}=\sum_{\ell=0}^{M}\frac{B_{\ell}(-b t)^{\ell}}{\ell !}\sum_{k\le t^{2\delta/3-1}}\frac{e^{-k(bcX+a+bd)t}}{k^{2-\ell}}+R_2(X,t),
\end{equation*}
where as $t\rrw 0^+$,
\[
R_2(X,t)\ll\sum_{k\le t^{2\delta/3-1}}\frac{e^{-kbcXt}}{k^{2}}\sum_{\ell> M}(kb t)^{\ell}\ll |bt^{2\delta/3}|^{M+1}\ll t^{\frac{\delta}{2}( M+1)}.
\]
Using \eqref{PLF}, it is not difficult seen that
\[
\sum_{k\le t^{2\delta/3-1}}\frac{e^{-k(bcX+a+bd)t}}{k^{2}}\frac{-kb t}{e^{-kb t}-1}=\sum_{\ell=0}^{M}\frac{B_{\ell}(-b t)^{\ell}}{\ell !}\li_{2-\ell}\left(e^{-(bcX+a+bd)t}\right)+O\left(t^{\frac{\delta}{2}( M+1)}\right).
\]
By Lemma \ref{lm2}
\[\frac{\,d^M\li_{2-\ell}(e^{-bcXt}e^{x})}{\,d x^M}=\li_{2-\ell-M}\left(e^{-bcXt}e^x\right).\]
Note the fact that
\begin{align*}
\li_{1-m}\left(e^{-bcXt}\right)\ll \left(\sum_{k\le t^{2\delta/3-1}}+\sum_{k> t^{2\delta/3-1}}\right)k^{m-1}e^{-kbcXt}\ll t^{(2\delta/3-1)m}
\end{align*}
holds for $m\le |\log t|$, then Taylor Theorem implies that
\[
\li_{2-\ell}\left(e^{-(bcX+a+bd)t}\right)=\sum_{k=0}^{M}\frac{\li_{2-\ell-k}\left(e^{-bcXt}\right)}{k!}(-(a+bd)t)^k+R_3(X,t),
\]
where
\begin{align*}
R_3(X,t)&=\int\limits_{0}^{-(a+bd)t}\li_{2-\ell-M-1}\left(e^{-bcXt}e^x\right)\frac{(-(a+bd)t-x)^M}{M!}\,dx\\
&\ll \frac{((a+bd)t)^{M+1}}{(M+1)!}\li_{1-\ell-M}\left(e^{-bcXt}\right)\ll t^{ 2M\delta/3}t^{(2\delta/3-1)\ell}.
\end{align*}
for $\ell\le M$ as $t\rrw 0^+$. Therefore, we obtain that
\begin{align*}
&\sum_{k\le t^{2\delta/3-1}}\frac{e^{-k(bcX+a+bd)t}}{k^{2}}\frac{-kb t}{e^{-kb t}-1}\\
&\qquad =\ssum_{0\le k,\ell\le M}\frac{B_{\ell}(-b t)^{\ell}}{\ell !}\frac{\li_{2-\ell-k}\left(e^{-bcXt}\right)}{k!}(-(a+bd)t)^k+O\left(t^{\frac{\delta}{2}( M+1)}\right).
\end{align*}
Thus we obtain that as $t\rrw 0^+$,
\begin{align*}
\cf(X,t)=&Xv-AX^2t-BXt+O\left(t^{\frac{\delta}{2}( M+1)}+\sum_{M<k\le 2M}t^{k-1+(2\delta/3-1)(k-1)}\right)\\
&-\sum_{m=0}^{M}(-t)^{m}\sum_{a,b,c,d}\frac{\li_{2-m}(e^{-bc Xt})}{t}\sum_{0\le \ell\le m}\frac{B_{\ell}b^{\ell}(a+bd)^{m-\ell}\cs(a,b,c,d)}{b\ell !(m-\ell)!}.
\end{align*}
Namely,
\begin{align*}
\cf(X,t)=&\frac{1}{t}\left(v(Xt)-A(Xt)^2-\sum_{a,b,c,d}\li_{2}(e^{-bc Xt})b^{-1}\cs(a,b,c,d)\right)\\
&+\sum_{a,b,c,d}\li_{1}(e^{-bc Xt})\left(B_{0}(a+bd)+B_{1}b\right)\frac{\cs(a,b,c,d)}{b}-B(Xt)\\
&+\sum_{m=1}^{M}t^{m}(-1)^{m}\sum_{a,b,c,d}\li_{1-m}(e^{-bc Xt})\frac{b^m\cs(a,b,c,d)}{(m+1)!}B_{m+1}\left(\frac{a+bd}{b}\right)+o\left(t^{\delta M/2}\right)
\end{align*}
for each $M\in\nb$. Finally, note that if $X\gg 1/t$ then $\li_{1-m}(e^{-bc Xt})\ll 1$, thus we immediately get the proof of \eqref{fasy1}. Which completes the proof of the lemma.
\end{proof}

\section{The proof of the main theorem}\label{sec5}
Let us denote by $\ss_H$ the set of all local maximum points of $H_{-1}(x)$ with $x\in \rb_{>0}$ and let $t>0$ sufficiently small.  We first have $\ss_H$ is a finite point set or an empty set by the definition of $H_{-1}(u)$ in Lemma \ref{fasy}. We rewritten \eqref{sum} as
\begin{equation*}
\ch(z;q)=\left(\sum_{m\le f_m/t}+\sum_{f_m/t<m\le f_M/t}+\sum_{m>f_M/t}\right)e^{\cf(m,t)}:=\Sigma_1+\Sigma_2+\Sigma_3,
\end{equation*}
where $f_m$ and $f_M$ be defined as follows:
\begin{enumerate}\label{condition}
\item[(\Rmnum{1}).]
If $\ss_H$ is nonempty, then set
\[f_m=(\min_{u\in\ss_H}u)-1/|\log t|\; \mbox{and}\; f_M=(\max_{u\in\ss_H}u)+1/|\log t|.\]
\item[(\Rmnum{2}).]
If $\ss_H$ is an empty set, then set $f_m=f_M=1$.
\end{enumerate}
We also write by
\begin{equation*}
\mathcal{I}_{\ch}(z;q)=\int\limits_{0}^{\infty}e^{\cf(x,t)}\,d x=\left(\int\limits_{0}^{f_m/t}+\int\limits_{f_m/t}^{f_M/t}+\int\limits_{f_M/t}^{\infty}\right)e^{\cf(x,t)}\,dx:=\mathcal{I}_{1}+\mathcal{I}_{2}+\mathcal{I}_{3}.
\end{equation*}
Note that if $\ss_H$ is nonempty, then by Lemma \ref{fasy} and Lemma \ref{th1}, as $t\rrw 0^+$
\begin{equation}\label{Sigma2}
\Sigma_2=(1+o(t^p))\mathcal{I}_2
\end{equation}
holds for each $p\ge 0$. Thus we just need consider the estimates of $\Sigma_1, \mathcal{I}_1, \Sigma_3$ and $\mathcal{I}_3$.
\subsection{The treat of $\Sigma_1$ and $\mathcal{I}_1$}\label{sec51}
\begin{lemma}\label{l51}
Let $\cf(m,t)$ be defined as \eqref{cf1} and $f_m$ be defined as above. Then, we have as $t\rrw 0^+$
\[\Sigma_1\ll \exp\left(\cf(f_m/t,t)+|\log t|^3\right)\;\mbox{and}\;\mathcal{I}_1\ll \exp\left(\cf(f_m/t,t)+|\log t|^3\right).\]
\end{lemma}
\begin{proof}
First of all, let $t\rrw 0^+$ and $x\ge 0$. It is not difficult to compute that
\begin{align*}
\frac{\P \cf(x,t)}{\P x}&=v-2Axt-Bt-t\sum_{\alpha,\beta,\gamma}\alpha S_{\alpha\beta\gamma}\sum_{k=1}^{\infty}\frac{e^{-k(\alpha x+\gamma-\beta)t}}{e^{k\beta t}-1}\\
&=v-2Axt-t\sum_{\alpha,\beta,\gamma}\alpha S_{\alpha\beta\gamma}\sum_{k=1}^{\left\lfloor\frac{1}{2\beta t}\right\rfloor}\frac{e^{-k(\alpha x+\gamma)t}}{1-e^{-k\beta t}}+O\left(\exp\left(-\frac{\alpha}{2\beta}x\right)\right)\\
&=v-2Axt-t\sum_{\alpha,\beta,\gamma}\alpha S_{\alpha\beta\gamma}\sum_{k=1}^{\left\lfloor\frac{1}{2\beta t}\right\rfloor}\frac{e^{-k(\alpha x+\gamma)t}}{k\beta t}+O\left(\frac{1}{1+x}+t\right)\\
&=v-2Axt-\sum_{\alpha,\beta,\gamma}\frac{\alpha}{\beta} S_{\alpha\beta\gamma}\sum_{k=1}^{\infty}\frac{e^{-k\alpha( x+\gamma/\alpha)t}}{k}+O\left(\frac{1}{1+x}+t\right).
\end{align*}
Therefore it is not hard seen that
\begin{equation}\label{1dr}
\frac{\P \cf(x,t)}{\P x}=\begin{cases}H_{-1}'(xt)+O(1/x+t)\quad &x\ge 1\\
O(|\log t|)& x\in[0,2].
\end{cases}
\end{equation}
Thus for each $X\in [1, f_m/t]$ and $Y\gg 1/t$, we have
\begin{align*}
\cf(X,t)&=\cf(Y,t)-\int\limits_{X}^{Y}\left(H_{-1}'(xt)+O(1/x+t)\right)\,dx\\
&=\cf(Y,t)-H_{-1}(Yt)/t+H_{-1}(Xt)/t+O\left(|\log t|\right).
\end{align*}
Using Lemma \ref{fasy} we obtain that
\[\cf(X,t)=H_{-1}(Xt)/t+O(|\log t|).\]
Therefore as $t\rrw 0^+$, we have
\[\exp\left(\cf(x,t)\right)\ll \exp\left(H_{-1}(xt)/t+|\log t|^2\right)\]
Moreover, from the assumption of Theorem \ref{mt} and Lemma \ref{fasy} we find that
$$H_{-1}(xt)\le H_{-1}(f_m)=\cf(f_m/t,t)+O(1)$$
for each $x\in [0, f_m/t]$. By \eqref{1dr} we have $\cf(X,t)\ll \cf(1,t)+|\log t|$ for $X\in[0,1]$. Hence
\begin{align*}
\Sigma_1\ll\sum_{m\le f_m/t}\exp\left(H_{-1}(mt)/t+O(|\log t|)\right)\ll\exp\left(\cf(f_m/t,t)+|\log t|^3\right)
\end{align*}
and
\begin{align*}
\mathcal{I}_1&\ll\int\limits_{0}^{f_m/t}\exp\left(H_{-1}(xt)/t+O(|\log t|)\right)\,dx\ll \exp\left(\cf(f_m/t,t)+|\log t|^3\right).
\end{align*}
Which completes the proof of the lemma.
\end{proof}
\subsection{The estimate of $\Sigma_3$ and $\mathcal{I}_3$}\label{sec52}
If $A=v=0$ then
\[\ch(z;q)=\ch(1;q)=\sum_{m\in\nb}\frac{q^{Bm}}{\prod_{\alpha,\beta,\gamma}(q^{\alpha m+\gamma};q^{\beta})_{\infty}^{S_{\alpha\beta\gamma}}}\]
and for $x\gg 1/t$,
\begin{equation}\label{0dr}
\cf(x,t)=\frac{1}{t}\sum_{\alpha}\li_2(e^{-\alpha xt})\sum_{\alpha,\beta,\gamma} \beta^{-1}S_{\alpha\beta\gamma}-Bxt+O\left(\sum_{\alpha}e^{-\alpha xt}\right)+O(t)
\end{equation}
and
\begin{equation}\label{1dr}
\frac{\P\cf(x,t)}{\P x}=-\sum_{\alpha}\alpha\li_1(e^{-\alpha xt})\sum_{\alpha,\beta,\gamma} \beta^{-1}S_{\alpha\beta\gamma}-Bt+O\left(\sum_{\alpha}te^{-\alpha xt}\right)+O(t^{2})
\end{equation}
by Lemma \ref{fasy}. The assumption of Theorem \ref{mt} implies that
$$H_{-1}(u)=\sum_{\alpha}\li_2(e^{-\alpha u})\sum_{\beta,\gamma} \beta^{-1}S_{\alpha\beta\gamma}\not\equiv 0.$$
and hence
$\sum_{\beta,\gamma} \beta^{-1}S_{\alpha\beta\gamma}\neq 0$
for some $\alpha$. Therefore we can rewritten \eqref{1dr} as
\begin{equation*}
\frac{\P\cf(x,t)}{\P x}=-\sum_{k=1}^H\alpha_k\log(1-e^{-\alpha_k xt})f(\alpha_k)-Bt+O\left(te^{-\alpha_1xt}+t^{2}\right).
\end{equation*}
Furthermore,
\begin{equation}\label{1dr2}
\frac{\P\cf(x,t)}{\P x}=e^{-\alpha_1 xt}\alpha_1f(\alpha_1)\left(1+O\left(e^{-\alpha_1xt}+e^{-(\alpha_2-\alpha_1)xt}\right)\right)-Bt(1+O(t)).
\end{equation}
This yields if $f(\alpha_1)>0$ then $\P \cf(x,t)/\P x=0$ has only one solution $X_{\ch}$ for $x>C/t$ with $C>0$ be sufficiently large depends only on $\ch(1;\cdot)$. Moreover, it is clear to compute that
\begin{equation*}
X_{\ch}=\frac{1}{\alpha_1 t}\left(\log\left(\frac{1}{t}\right)+O(1)\right).
\end{equation*}
In this case, we have the following lemma.
\begin{lemma}\label{l52} If $f(\alpha_1)>0$ then as $t\rrw 0^+$
\begin{equation*}
\Sigma_3=\left(1+o(t^p)\right)\mathcal{I}_3+O\left(\exp\left(\cf(f_M/t,t)+2|\log t|\right)\right)
\end{equation*}
holds  for each $p\ge 0$.
\end{lemma}
\begin{proof}
The proof of this lemma is similar with the proof of \eqref{T}. We first give the estimates of the derivatives of $\cf(x,t)$ for $x>A_h|\log t|/t$ with $A_h=1/(2+2\alpha_1)$. From \eqref{0dr} we have
\begin{align*}
\cf(x,t)&=-\frac{f(\alpha_1)}{t}e^{-\alpha_1 xt}\left(1+O\left(e^{-\alpha_1 xt}+e^{-(\alpha_2-\alpha_1) xt}\right)\right)-Bxt+O(e^{-\alpha_1xt})\\
&=-\frac{f(\alpha_1)}{t}e^{-\alpha_1 xt}\left(1+o(1)\right)-Bxt.
\end{align*}
From \eqref{1dr2} we have
\[
\frac{\P\cf(x,t)}{\P x}\ll e^{-\alpha_1 xt}+t\ll t^{A_h}.
\]
From Lemma \ref{fasy} we have for each $N\ge 2$,
\[
\frac{\P^N\cf(x,t)}{\P x^N}\ll t^{N-1}\ll t^{A_hN}.\]
Thus similar with \eqref{0be} and \eqref{Nbe}, for each $N\in\nb$ we have
\begin{align}\label{bee1}
\frac{\P^N e^{\cf(x,t)}}{\P x^N}\bigg|_{\lceil A_h |\log t|/t\rceil}^{\infty}&\ll e^{\cf(\lceil A_h |\log t|/t\rceil,t)}\ll \exp\left(-\frac{f(\alpha_1)}{t}e^{-\alpha_1 \lceil A_h |\log t|/t\rceil t}\left(1+o(1)\right)\right)\nonumber\\
&\ll \exp\left(-f(\alpha_1)t^{-1+\alpha_1A_h}\left(1+o(1)\right)\right)\ll\exp\left(-1/\sqrt{t}\right).
\end{align}
Similar with \eqref{EFe}, we have for each $N\in\nb_1$,
\[
\frac{\P^N e^{\cf(x,t)}}{\P x^N}\ll e^{\cf(x,t)}\sum_{\substack{m_1,m_2,\dots m_N\in\nb\\m_1+2m_2+\dots+Nm_N=N}}\prod_{j=1}^{N}\left(t^{jA_h}\right)^{m_j}\ll t^{A_hN}e^{\cf(x,t)}.
\]
Further more, it is not not difficult compute that
\begin{align}\label{IE}
\int\limits_{\lceil A_h |\log t|/t\rceil}^{\infty}e^{\cf(x,t)}\,dx&=\int\limits_{\lceil A_h |\log t|/t\rceil}^{\infty}\exp\left(-Bxt-\frac{f(\alpha_1)}{t}e^{-\alpha_1xt}\left(1+o(1)\right)\right)\,dx\nonumber\\
&=\frac{1}{Bt}\int\limits_{0}^{t^{A_hB}(1+o(1))}\exp\left(-\frac{f(\alpha_1)}{t}y^{\alpha_1/B}(1+o(1))\right)\,dy\asymp t^{\alpha_1/B-1}.
\end{align}
Thus similar with \eqref{T} and the using of Euler--Maclaurin formula yields that
\begin{equation}\label{SI3}
\sum_{m> \lceil A_h |\log t|/t\rceil}e^{ \cf(m,t)}=\left(1+o(t^p)\right)\int\limits_{\lceil A_h |\log t|/t\rceil}^{\infty}e^{\cf(x,t)}\,dx
\end{equation}
holds for each $p\ge 0$.
On the other hand, it is not difficult seen that $\cf(x,t)$ has an unique minimal point on $x\in[f_M/t,A_h|\log t|/t]$ if $\ss_H$ is nonempty or $\cf(x,t)$ is an increasing function if $\ss_H$ is empty. Thus
\begin{align}\label{SIE}
\sum_{f_M/t<m\le \lceil A_h |\log t|/t\rceil}e^{ \cf(m,t)}&\ll t^{-2}\exp\left(\cf(f_M/t,t)\right)+t^{-2}\exp\left(\cf(\lceil A_h |\log t|/t\rceil,t)\right)\nonumber\\
&\ll \exp\left(\cf(f_M/t,t)+2|\log t|\right) +\exp\left(-1/\sqrt{t}\right)
\end{align}
and
\begin{align}\label{RREE}
\mathcal{I}_3-\int\limits_{\lceil A_h |\log t|/t\rceil}^{\infty}e^{\cf(x,t)}\,dx&\ll t^{-2}\left(\exp\left(\cf(f_M/t,t)\right)+\exp\left(\cf(\lceil A_h |\log t|/t\rceil,t)\right)\right)\nonumber\\
&\ll \exp\left(\cf(f_M/t,t)+2|\log t|\right) +\exp\left(-1/\sqrt{t}\right)
\end{align}
by \eqref{bee1}.
Combining \eqref{SI3}, \eqref{IE}, \eqref{SIE} and \eqref{RREE} we immediately obtain the proof of the lemma.
\end{proof}

We have the following estimate for $\Sigma_3$ and $\mathcal{I}_3$.
\begin{lemma}\label{l53}
Let $\cf(m,t)$ be defined as \eqref{cf} and $f_M$ be defined as above. If $A>0$ or $v<0$ or $f(\alpha_1)<0$, then as $t\rrw 0^+$,
\[
\Sigma_3\ll \exp\left(\cf(f_M/t,t)+4|\log t|\right) \;\mbox{and}\; \mathcal{I}_3\ll \exp\left(\cf(f_M/t,t)+4|\log t|\right).
\]
\end{lemma}
\begin{proof}
If $x \ge 1/t^4$ then as $t\rrw 0^+$,
\begin{align}\label{mm3}
\cf(x,t)&=xv-Ax^2t-Bxt-\sum_{\alpha,\beta,\gamma}S_{\alpha\beta\gamma}\sum_{k=1}^{\infty}\frac{e^{-k(\alpha x+\gamma )t}}{k(e^{k\beta t}-1)}\nonumber\\
&<xv- Ax^2t-Bxt+\sum_{\alpha,\beta,\gamma}|S_{\alpha\beta\gamma}|\sum_{k=1}^{\infty}\frac{e^{-k\alpha xt}}{k^2\beta t}\nonumber\\
&=xv- Ax^2t-Bxt+O(e^{-1/t^2}).
\end{align}
If $f(\alpha_1)<0$, then from \eqref{1dr2} we have $\P \cf(x,t)/\P x<0$ for $x>C_1/t$ with $C_1>0$ be sufficiently large depends only on $\ch(1;\cdot)$. On the other hand, by \eqref{1dr} if $x\gg 1/t$ then
\[
\frac{\P \cf(x,t)}{\P x}=v-2Axt+\sum_{\alpha}\alpha \log(1-e^{-\alpha xt})\sum_{\beta,\gamma}\beta^{-1}S_{\alpha\beta\gamma}+O\left(1/x+t\right).
\]
Thus $A>0$ or $v<0$ implies that ${\P \cf(x,t)}/{\P x}<0$ holds for $x>C_2/t$ for some $C_2>0$ sufficiently large depends only on $\ch(z;\cdot)$. Therefore under the condition of the lemma, the definition of $f_M$ and the estimate of $r_u(t)$ in Lemma \ref{mainlemma} for $\cf(x,t)$ implies that $\P \cf(x,t)/\P x<0$ for all $x>f_M/t$. Hence
\begin{align*}
\Sigma_3=&\sum_{f_M/t<m\le 1/t^4}e^{\cf(m,t)}+\sum_{m> 1/t^4}e^{\cf(m,t)}\\
&\ll \exp\left(\cf(f_M/t,t)+4|\log t|\right)+\sum_{m\ge 1/t^4}e^{mv- Am^2t-Bmt}\\
&\ll \exp\left(\cf(f_M/t,t)+4|\log t|\right)+e^{-1/t^2}.
\end{align*}
Moreover, Lemma \ref{fasy} implies that $\cf(x,t)\ll 1/t$ for all $x\ge 1/t$, hence
\[\Sigma_3 \ll \exp\left(\cf(f_M/t,t)+4|\log t|\right).\]
Now, the estimate for $\mathcal{I}_3$ is easy to establish. Thus we complete the proof of the lemma.
\end{proof}
\subsection{The final estimate for $\ch(z;q)$}\label{sec53}
From Lemma \ref{taylor}, we  have for each $u\in\ss_H$, as $t\rrw 0^+$,
\begin{align*}
\cf(u/t\pm 1/(t|\log t|),t)&=\cf(u/t,t)+\frac{1}{(2k_{u})!}\frac{\P^{2k_{u}}\cf}{\P x^{2k_{u}}}\left(u/t,t\right)\frac{1}{t|\log t|^{2k_{u}}}\left(1+o(1)\right)\\
&< \cf(u/t,t)-2/\sqrt{t}.
\end{align*}
Therefore if $\ss_H$ is nonempty then
\begin{equation*}
\mathcal{I}_2\gg \sum_{u\in\ss_H}e^{\cf(u/t,t)}/t
\end{equation*}
by Lemma \ref{th2}. Therefore from Lemma \ref{l51} we have for each $p\ge 0$, as $t\rrw 0^+$,
\begin{equation*}
\Sigma_1\ll t^p\mathcal{I}_2,~\mathcal{I}_1\ll t^p\mathcal{I}_2.
\end{equation*}
From Lemma \ref{l52} and Lemma \ref{l53} we have:
Let $f(\alpha_1)$ be defined as \eqref{f1}. If $A=v=0$ and $f(\alpha_1)>0$ then
\begin{equation*}
\Sigma_3=\left(1+o(t^p)\right)\mathcal{I}_3+o(t^p\mathcal{I}_2).
\end{equation*}
If $A>0$ or $v<0$ or $f(\alpha_1)<0$, then
\begin{equation*}
\Sigma_3\ll t^p\mathcal{I}_2 \;\mbox{and}\; \mathcal{I}_3\ll t^p\mathcal{I}_2.
\end{equation*}
Thus together with \eqref{Sigma2}, we obtain that for each $p\ge 0$, as $t\rrw 0^+$,
\[\ch(z;q)=(1+o(t^p))\mathcal{I}_{\ch}(z;q).\]
If $\ss_H$ is empty then $f(\alpha_1)>0$, thus from Lemma \ref{l52} we have as $t\rrw 0^+$,
\begin{align*}
\ch(z;q)&=\Sigma_1+\Sigma_3+\mathcal{I}_1-\mathcal{I}_1\\
&=O\left(\exp\left(\cf(1/t,t)+|\log t|^3\right)\right)+\left(1+o(t^p)\right)\mathcal{I}_3+\mathcal{I}_1\\
&=\left(1+o(t^p)\right)\mathcal{I}_{\ch}(z;q)+o(t^p\mathcal{I}_3)=\left(1+o(t^p)\right)\mathcal{I}_{\ch}(z;q)
\end{align*}
holds  for each $p\ge 0$.

Which completes the proof of the Theorem \ref{mt}.
\subsection{The proof of Theorem \ref{mc}}
From above discussion, it is not difficult seen that for each $p\ge 0$, as $t\rrw 0^+$
\begin{equation}\label{aspa}
\ch(z;q)=(1+o(t^p))\left(\Sigma_2+\int\limits_{\log|\log t|/t}^{\infty}e^{\cf(x,t)}\,dx\right),
\end{equation}
where if $\ss_H$ is empty then $\Sigma_2=0$. If $\ss_H$ is nonempty then we can obtain the full asymptotics by Lemma \ref{th2}. Thus we just
need to consider the integration part of above. On the other hand, combining Subsection \ref{sec41}, the full asymptotic expansion for $\sh(z;q)$ immediately follows the full asymptotic expansion of $\ch(z;q)$.
We first give the results of $\Sigma_2$.
\subsubsection{Asymptotic expansion for $\Sigma_2$} Let denote by $m_u$ the minimum positive integer such that
$H_{-1}^{\langle 2m_u\rangle}(u)\neq 0$.
\begin{equation}\label{Vh}
V_{\ch}(u,t)=\left(\frac{-1}{(2m_u) !}\frac{\P^{2m_u}\cf}{\P x^{2m_u}}(u/t,t)\right)^{\frac{1}{2m_u}},
\end{equation}
\begin{align}\label{mu}
\mu_{\ell}(u,t)=\sum_{\substack{\ell_r\in\nb\\ \sum_{r\ge 1}^*r\ell_r=\ell}}\prodstar_{r\ge 1}\frac{1}{\ell_r!}\left(\frac{[(2m_u)!]^{r/(2m_u)}}{r !}\frac{\frac{\P^{r}\cf}{\P x^{r}}(u/t,t)}{\left|\frac{\P^{2m_u}\cf}{\P x^{2m_u}}(u/t,t)\right|^{r/(2m_u)}}\right)^{\ell_r}.
\end{align}
where $*$ means that $r\neq 2m_u$ and from Subsection \ref{pt12} the estimate
\begin{align*}
\mu_{\ell}(u,t)\ll t^{\frac{\ell}{2m_u(2m_u+1)}}
\end{align*}
holds for each $\ell\in\nb$.
From Lemma \ref{th2} we have the full asymptotic expansion in $\mu_{2\ell}(u,t)$ of the form
\begin{equation}\label{formula}
\Sigma_2\sim \sum_{
u\in\ss_H}\frac{\exp(\cf(u/t,t))}{V_{\ch}(u,t)}\sum_{\ell=0}^{\infty}\Gamma\left(\frac{2\ell+1}{2m_u}\right)\frac{\mu_{2\ell}(u,t)}{m_u},
\end{equation}
\subsubsection{Asymptotic expansion for the integration part}
Now we consider the integral of \eqref{aspa}. In fact, we shall consider the case is when $A=v=0$ with $f(\alpha_1)>0$. We have first
\begin{align*}
I_{\ch}(t):&=\int\limits_{\log|\log t|/t}^{\infty}e^{\cf(x,t)}\,dx\\
&=\int\limits_{\log|\log t|/t}^{\infty}\exp\left(-Bxt+\sum_{k=1}^{\infty}\sum_{\ell=1}^{H}e^{-k\alpha_{\ell} xt}\sum_{\beta,\gamma}\frac{S_{\alpha_{\ell}\beta\gamma}e^{-k\gamma t}}{k(1-e^{-\beta kt})}\right)\,d x
\end{align*}
by \eqref{cf1}. We denote by
\[h_{k,\alpha_{\ell}}(t):=-\sum_{\beta,\gamma}\frac{S_{\alpha_{\ell}\beta\gamma}e^{-k\gamma t}}{k(1-e^{-\beta kt})}.\]
Then for each $k\in\nb_1$ and $\alpha_{\ell}$, it is obvious that
\[h_{k,\alpha_{\ell}}(t)=\frac{1}{k^2t}\left(f(\alpha_{\ell})-\sum_{r\ge 1}\frac{(-kt)^r}{r!}\sum_{s=0}^r\binom{r}{s}\frac{B_s}{\beta}\sum_{\beta,\gamma}\beta^s\gamma^{r-s}S_{\alpha_{\ell}\beta\gamma}\right).\]
Further more, we obtain that
\begin{align*}
I_{\ch}(t)
=&\frac{1}{Bt}\int\limits_{0}^{1/|\log t|^B}\exp\left(-\sum_{k=1}^{\infty}\sum_{\ell=1}^{H}y^{k\alpha_{\ell}/B} h_{k,\alpha_{\ell}}(t)\right)\,d y\\
=&\frac{1+o(t^p)}{Bt[h_{1,\alpha_1}(t)]^{\frac{B}{\alpha_1}}}\int\limits_{0}^{1/(t|\log t|)}u^{\frac{B}{\alpha_1}-1}e^{-u}\exp\left(-\ssum_{\substack{(\ell,k)\in\nb_1^2\setminus\{(1,1)\}\\ \ell\le H}}\frac{u^{\frac{k\alpha_{\ell}}{\alpha_1}}h_{k,\alpha_{\ell}}(t)}{[h_{1,\alpha_1}(t)]^{\frac{k\alpha_{\ell}}{\alpha_1}}}\right)\,d y
\end{align*}
for each $p\ge 0$. Moreover, for each $k\in\nb$, we have
\begin{equation}\label{kpe}
\frac{h_{k,\alpha_{\ell}}(t)}{[h_{1,\alpha_1}(t)]^{\frac{k\alpha_{\ell}}{\alpha_1}}}=\frac{f(\alpha_{\ell})}{k^2[f(\alpha_1)]^{\frac{k\alpha_{\ell}}{\alpha_1}}}t^{\frac{k\alpha_{\ell}}{\alpha_1}-1}\left(1+\sum_{r\ge 1}b_r(\alpha_1,\alpha_{\ell},k)t^r\right)
\end{equation}
holds for some $b_r(\alpha_1,\alpha_{\ell},k)\in \rb (r\ge 1)$ as $t\rrw 0^+$. In odder to give the asymptotic expansion, we shall define the countable set
\begin{equation}\label{Lambda}
\Lambda(\ch):=\left\{\lambda_1+\sum_{\ell=2}^{H}\sum_{k=1}^{\infty}\lambda_{\ell k}(k\alpha_{\ell}/\alpha_1-1): \lambda_1,\lambda_{\ell k}\in\nb, 2\le \ell\le H, k\ge 1\right\}.
\end{equation}
Then, we formally have
\begin{equation}\label{Kappa}
\exp\left(-\ssum_{\substack{(\ell,k)\in\nb_1^2\setminus\{(1,1)\}\\ \ell\le H}}\frac{u^{k\alpha_{\ell}/\alpha_1 }h_{k,\alpha_{\ell}}(t)}{[h_{1,\alpha_1}(t)]^{k\alpha_{\ell}/\alpha_1}}\right)=\sum_{\lambda\in\Lambda(\ch)}\kappa_{\lambda}(t)u^{\lambda}
\end{equation}
with some $\kappa_{\lambda}(t)\in\rb$.
Moreover, for each $\lambda\in\Lambda(\ch)$, by \eqref{kpe} and \eqref{Kappa}
\begin{align*}
\kappa_{\lambda}(t)&\ll \sum_{\substack{\ell_1+\dots+\ell_H=\lambda\\ \ell_j\in \Lambda(\ch), j=1,\dots,H }}\left(\sum_{\substack{\ell_{1r}\ge 2\\ 2\ell_{12}+3\ell_{13}+\dots=\ell_1}}t^{\sum_{r\ge 2}(r-1)\ell_{1r}}\right)\prod_{j=2}^H\left(\sum_{\substack{\ell_{jr}\ge 1\\ \ell_{j1}+2\ell_{j2}+\dots=\ell_j}}t^{\sum_{r\ge 1}\left(\frac{\alpha_{j}}{\alpha_1}r-1\right)\ell_{jr}}\right)\\
&\ll \sum_{\substack{\ell_1+\dots+\ell_H=\lambda\\ \ell_j\in\Lambda(\ch), j=1,\dots,H }}t^{\ell_1/2}\prod_{j=2}^Ht^{\left(\frac{\alpha_{j}}{\alpha_1}-1\right)\ell_j}=\sum_{\substack{\ell_1+\dots+\ell_H=\lambda\\ \ell_j\in\Lambda(\ch), j=1,\dots,H }}t^{\lambda/2+\sum_{j=2}^H\left(\alpha_j/\alpha_1-3/2\right)\ell_j}.
\end{align*}
Thus it is not hard to see that
\[\kappa_{\lambda}(t)\ll t^{\lambda/2}+t^{(\alpha_2/\alpha_1-1)\lambda}.\]
We choose
$$\alpha(h)=\frac{1}{4+2/(\alpha_2/\alpha_1-1)},$$
then it not difficult seen that
\begin{align*}
I_{\ch}(t)
=&\frac{1+o(t^p)}{\alpha_1t[h_{1,\alpha_1}(t)]^{\frac{B}{\alpha_1}}}\int\limits_{0}^{t^{-\alpha(h)}}u^{\frac{B}{\alpha_1}-1}e^{-u}\exp\left(-\ssum_{\substack{(\ell,k)\in\nb_1^2\setminus\{(1,1)\}\\ \ell\le H}}\frac{u^{\frac{k\alpha_{\ell}}{\alpha_1}}h_{k,\alpha_{\ell}}(t)}{[h_{1,\alpha_1}(t)]^{\frac{k\alpha_{\ell}}{\alpha_1}}}\right)\,d u
\end{align*}
and furthermore,
\begin{equation}\label{formula2}
I_{\ch}(t)=\frac{1+o(t^p)}{V(t)}\int\limits_{0}^{t^{-\alpha(h)}}y^{\frac{B}{\alpha_1}-1}e^{-y}\sum_{\lambda\in\Lambda(\ch)}\kappa_{\lambda}(t)y^{\lambda}\,d y\sim \sum_{\lambda\in\Lambda(\ch)}\frac{\kappa_{\lambda}(t)}{V(t)}\Gamma\left(\frac{B}{\alpha_1}+\lambda\right).
\end{equation}
where
\[V(t)=\alpha_1t[h_{1,\alpha_1}(t)]^{\frac{B}{\alpha_1}}.\]

\section{Some applications of the main results}\label{sec6}
\subsection{Asymtotics of the basic hypergeometric series}
We write ${\bf a}=(a_1,a_2,...,a_r)\in \rb_{>0}^{r} $, ${\bf b}=(b_1,b_2,...,b_s)\in\rb_{>0}^{s}$ with $\ell:=s-r>0$. Assuming $v\in \rb$ and $t>0$ and considering the following basic hypergeometric series:
\begin{equation*}
_{r}\phi_{s}({\bf a,b};t,v):=\sum_{k=0}^{\infty}\frac{(e^{-ta_{1}},\dotsc,e^{-ta_{r}};e^{-t})_{k}}{(e^{-tb_{1}},\dotsc,e^{-tb_{s}};e^{-t})_{k}}e^{-\ell t\binom{k}{2}}e^{vk}.
\end{equation*}
Applying Theorem \eqref{mt}, we define
\begin{equation*}
\Phi_{m}(u)=\begin{cases}
vu-\ell\left(u^2/2+\li_{2}(e^{-u})\right) &m=-1 \\
\left(\sum\limits_{\nu=1}^{s}\left(b_{\nu}-\frac{1}{2}\right)-\sum\limits_{\mu=1}^{r}\left(a_{\mu}-\frac{1}{2}\right)\right)\li_{1}(e^{- u})+\frac{s-r}{2}u &m=0.
\end{cases}
\end{equation*}
Then
\[\Phi_{-1}'(u)=v-\ell\log(e^{u}-1),\]
then $\lim_{u\rrw 0^+}\Phi_{-1}'(u)=+\infty$ and hence $_{r}\phi_s$ satisfy the assumption of Theorem \ref{mt}. Therefore we obtain that for each $p\ge 0$, as $t\rrw 0^+$,
\begin{equation*}
_{r}\phi_{s}({\bf a,b};t,v)=(1+o(t^p))\int\limits_{0}^{\infty}\frac{(e^{-ta_{1}},\dotsc,e^{-ta_{r}};e^{-t})_{x}}{(e^{-tb_{1}},\dotsc,e^{-tb_{s}};e^{-t})_{x}}e^{-\ell t\binom{x}{2}}e^{vx}\,dx.
\end{equation*}
Furthermore, $u=\log(1+e^{v/\ell})$ is only one solution of $\Phi_{-1}'(u)=0$ and
\[\Phi_{-1}''(\log(1+e^{v/\ell}))=-\ell(1+e^{-v/\ell})<0.\]
Thus, by Theorem \ref{mc} we have
\[_{r}\phi_{s}({\bf a,b};t,v)\sim C_{\phi}t^{B_{\phi}}\exp\left(\frac{A_{\phi}}{t}\right)\left(1+\sum_{j=1}^{\infty}C_{j}t^{j}\right)\]
with
\[A_{\phi}=\frac{\ell}{2}\left( \frac{2v}{\ell}\log(1+e^{v/\ell})-\log^2(1+e^{v/\ell})+\frac{\pi^2}{3}-2\li_2\left(\frac{1}{1+e^{v/\ell}}\right)\right),\]
\[B_{\phi}=\sum_{\nu=1}^{s}b_{\nu}-\sum_{\mu=1}^{r}a_{\mu}-\frac{\ell+1}{2},\]
\[
C_{\phi}=\frac{\left({2\pi}\right)^{\frac{1-\ell}{2}}\left(\frac{1+e^{v/\ell}}{1+e^{-v/\ell}}\right)^{\ell/2}}{\sqrt{\ell}}
\frac{\prod_{\nu=1}^{s}\Gamma(b_{\nu})}{\prod_{\mu=1}^{r}\Gamma(a_{\mu})}\left(1+e^{-v/\ell}\right)^{\sum_{\nu=1}^{s}b_{\nu}-\sum_{\mu=1}^{r}a_{\mu}-\frac{1}{2}}.
\]
and for $j\in\nb_1$, the coefficients $C_{j}\in\rb$ are constant depends only on $_{r}\phi_{s}({\bf a,b};\cdot,v)$. In particular,
\[
_{r}\phi_{s}({\bf a,b};t,0)\sim \frac{1}{\pi^{\ell/2}}\sqrt{\frac{2\pi}{\ell}}\frac{\prod_{\nu=1}^s\Gamma(b_{\nu})}{\prod_{\mu=1}^{r}\Gamma(a_{\mu})}(2t)^{\sum_{\nu=1}^sb_{\nu}-\sum_{\mu=1}^ra_{\mu}-\frac{\ell+1}{2}}\exp\left(\frac{\ell\pi^2}{12 t}\right).
\]

\subsection{Asymtotics of some simple Eulerian Series}
\subsubsection{A simple Eulerian Series}
Let $A, C,D,E>0$, $F\ge 0$, $G\in\nb_1$ and $B\in\rb$. We consider the following Eulerian series.
\begin{equation*}
\rc(q)=\sum_{m\in\nb}\frac{q^{Am^2+Bm}}{(q^C;q^D)_{Em+F}^G}=\frac{1}{(q^C;q^D)_{\infty}^G}\sum_{m\in\nb}q^{Am^2+Bm}(q^{DEm+C+DF};q^D)_{\infty}^G.
\end{equation*}
By \eqref{Hm2}, Lemma \ref{fasy}, Theorem \ref{mt} and Theorem \ref{mc}, we define
\[
R_{-1}(u)=-Au^2-\frac{G}{D}\li_{2}(e^{-DE u})\]
and
\[
R_0(u)=-Bu-G\li_{1}(e^{-DE u})\left(\frac{1}{2}-F-\frac{C}{D}\right)\]
to replace $H_{-1}(u)$ and $H_{0}(u)$, respectively. Then we have
\[R_{-1}'(u)=-2Au-GE\log(1-e^{-DEu})\]
and
\[ R_{-1}''(u)=-2A-\frac{DGE^2}{e^{DEu}-1}.\]
Thus $\lim_{u\rrw 0^+}R_{-1}'(u)=+\infty$, Therefore for each $p\ge 0$, as $q\rrw 1^-$,
\[
\rc(q)=\frac{1+o(|\log q|^p)}{(q^C;q^D)_{\infty}^G}\int\limits_{0}^{\infty}q^{Ax^2+Bx}(q^{DEx+C+DF};q^D)_{\infty}^G\,dx.
\]
Furthermore, $R_{-1}'(u)=0$ with $u>0$ equivalent to
\[
(e^{-u})^{A/(EG)}+(e^{-u})^{DE}-1=0.
\]
This equation just have one solution $\zeta_R$ on $(0,\infty)$ and $R_{-1}''(\zeta_R)<0$. Thus by Theorem \ref{mc} it is easy to prove that the leading asymptotics
\begin{equation*}
\rc(e^{-t})\sim C_{\rc}t^{B_{\rc}}\exp\left(\frac{A_{\rc}}{t}\right),
\end{equation*}
where
\[A_{\rc}=-A\zeta_R^2+\frac{G}{D}\left(\frac{\pi^2}{6}-\li_2(e^{-DE \zeta_R})\right),~B_{\rc}=\frac{CG}{D}-\frac{G+1}{2}\]
and
\[C_{\rc}=\left(\frac{1}{2\pi}\right)^{\frac{G-1}{2}}\Gamma\left(\frac{C}{D}\right)^{G}\frac{\exp\left(\left(\frac{A(F+C/D-1)}{E}-B\right)\zeta_R\right)}{\sqrt{2A+{DGE^2}e^{\left(DE-{2A}/({EG})\right)\zeta_R}}}.\]
\cite{MR2290758}
\subsubsection{Some examples on mock theta functions} We now apply our main result to some mock theta functions.  It is easy check that
there are about a half mock theta functions of the website \cite{MMF} can directly use Theorem \ref{mt} and Theorem \ref{mc} to obtain the complete asymptotic expansion.
Here just gives the illustration of the following two examples.
\begin{example} The following example is a mock theta function of Ramanujan, which were proved in Andrews \cite{MR814916} and Hickerson\cite{MR969247}.
\begin{align*}
F_0(q)&=\sum_{m\in\nb}\frac{q^{m^2}}{(q^{m+1};q)_m}\\
&=\sum_{m\in\nb}\frac{(q^{2m+1};q)_{\infty}}{(q^{m+1};q)_{\infty}}q^{m^2}.
\end{align*}
By \eqref{Hm1}, Lemma \ref{fasy}, Theorem \ref{mt} and Theorem \ref{mc}, we define
\[F(u)=-u^2-\li_2(e^{-2u})+\li_2(e^{-u}) \]
and
\[ F_{00}(u)=-\frac{1}{2}\log(1+e^{-u})\]
to replace $H_{-1}(u)$ and $H_{0}(u)$, respectively. Then we have
\[F'(u)=-2u-\log(1-e^{-u})-2\log(1+e^{-u})\]
and
\[F''(u)=1-\frac{1}{1-e^{-u}}-\frac{2}{1+e^{-u}}.\]
Then $\lim_{u\rrw 0^+}F_{-1}'(u)=+\infty$, thus as $q\rrw 1^-$
\begin{equation*}
F_0(q)=\left(1+o(|\log q|^p)\right)\int\limits_{0}^{\infty}\frac{(q^{2x+1};q)_{\infty}}{(q^{x+1};q)_{\infty}}q^{x^2}\,dx
\end{equation*}
holds for each $p\ge 0$. Moreover, $F_{-1}'(u)=0$ with $u>0$ equivalent to
\[(e^{-u})^3+2(e^{-u})^2-e^{-u}-1=0.\]
We have
\[\zeta_F:=-\log\left(\frac{2}{3} \sqrt{7} \cos \left(\frac{1}{3} \cos ^{-1}\left(-\frac{1}{2 \sqrt{7}}\right)\right)-\frac{2}{3}\right)=0.2206\dots\]
is the solution and $F''(\zeta_F)<0$. Thus we have the leading asymptotics
\begin{align*}
F_{0}(e^{-t})&\sim e^{F_{00}(\zeta_F)}\sqrt{\frac{-2}{F^{''}(\zeta_F)}}\Gamma\left(\frac{1}{2}\right)\sqrt{\frac{1}{t}}e^{F(\zeta_F)/t}\\
&=\left(\frac{1-e^{-\zeta_F}}{2-e^{-\zeta_F}+e^{-2\zeta_E}}\right)^{\frac{1}{2}}\sqrt{\frac{2\pi}{t}}\exp\left(\frac{1}{t}\left(\li_2(e^{-\zeta_F})-\zeta_F^2-\li_2(e^{-2\zeta_F})\right)\right).
\end{align*}
\end{example}
\begin{example} The following mock theta function were proved in Berndt and Chan \cite{MR2351377}.
\begin{align*}
\phi_{-}(q)&=\sum_{m\in\nb_1}\frac{q^m(-q;q)_{2m-1}}{(q;q^2)_m}\\
&=\frac{(-q;q)_{\infty}}{(q;q^2)_{\infty}}\sum_{m\in\nb_1}\frac{(q^{2m+1};q^2)_{\infty}}{(-q^{2m};q)_{\infty}}q^m.
\end{align*}
Clearly,
\[
\sum_{m\in\nb_1}\frac{(q^{2m+1};q^2)_{\infty}}{(-q^{2m};q)_{\infty}}q^m=q\sum_{m\in\nb}\frac{(q^{2m+3};q^2)_{\infty}(q^{2m+2};q)_{\infty}}{(q^{4m+4};q^2)_{\infty}}q^m
\]
By \eqref{Hm1} and Theorem \ref{mc}, we define
\[P(u)=(\li_2(e^{-4u})-3\li_2(e^{-2u}))/2 \]
to replace $H_{-1}(u)$. Then
\[\lim_{u\rrw 0^+}P_{-1}'(u)=\lim_{u\rrw 0^+}\left(-\log(1-e^{-2u})+2\log(1+e^{2u})\right)=+\infty,\]
thus Theorem \ref{mt} implies that for each $p\ge 0$, as $q\rrw 1^-$
\[
\phi_{-}(q)=\left(1+o(|\log q|^p)\right)q\frac{(-q;q)_{\infty}}{(q;q^2)_{\infty}}\int\limits_{0}^{\infty}\frac{(q^{2x+3};q^2)_{\infty}(q^{2x+2};q)_{\infty}}{(q^{4x+4};q^2)_{\infty}}q^x\,dx.
\]
Furthermore, it is not difficult to prove that the above
\[
\phi_{-}(q)=\left(1+o(|\log q|^p)\right)\frac{(-q;q)_{\infty}}{(q;q^2)_{\infty}}\int\limits_{0}^{1}\frac{(qy;q^2)_{\infty}}{(-y;q)_{\infty}}\frac{\,dy}{2\sqrt{y}}
\]
holds for each $p\ge 0$ as $q\rrw 1^-$. Moreover, from Theorem \ref{mc} it is easy to obtain the leading asymptotics :
\[
\phi_{-}(e^{-t})(t)\sim \frac{\exp\left(\pi^2/(6t)\right)}{2\sqrt{3\pi t}}.
\]
\end{example}

\section*{Acknowledgment}
The author would like to thank his advisor Zhi-Guo Liu for consistent encouragement and useful suggestions. The author would also like to thank Ruiming Zhang for useful comments and  suggestions.
%\bibliographystyle{unsrt}
%\bibliography{test}

\end{document}